\documentclass[reqno, 12pt]{amsart}

\usepackage[margin=1in]{geometry}
\usepackage[utf8]{inputenc}
\usepackage[british]{babel}
\usepackage[T1]{fontenc}
\usepackage{amsmath, amsfonts, dsfont, amssymb, amsthm, csquotes, mathtools, hyperref, biblatex, setspace}
\allowdisplaybreaks
\numberwithin{equation}{section}

\newtheorem{thm}{Theorem}[section]

\newtheorem{lem}[thm]{Lemma}

\theoremstyle{definition}

\newtheorem*{ack}{Acknowledgements}

\theoremstyle{remark}

\DeclareMathOperator{\GL}{GL}
\DeclareMathOperator{\ord}{ord}
\DeclareMathOperator{\tr}{Tr}
\DeclareMathOperator{\cha}{char}
\DeclareMathOperator{\mor}{Mor}
\DeclareMathOperator{\Q}{\mathbb{Q}}
\DeclareMathOperator{\PP}{\mathbb{P}}
\DeclareMathOperator{\C}{\mathbb{C}}
\DeclareMathOperator{\Z}{\mathbb{Z}}
\DeclareMathOperator{\R}{\mathbb{R}}
\DeclareMathOperator{\N}{\mathbb{N}}
\DeclareMathOperator{\T}{\mathbb{T}}

\newcommand{\abs}[1]{\left\lvert#1\right\rvert}

\newcommand{\floor}[1]{\left\lfloor#1\right\rfloor}
\newcommand{\legendre}[2]{\genfrac{(}{)}{}{}{#1}{#2}}
\newcommand{\fq}[0]{\mathbb{F}_q}
\newcommand{\fqt}[0]{\mathbb{F}_q [t]}
\newcommand{\fql}[0]{\mathbb{F}_q ((t^{-1}))}

\makeatletter
\DeclareRobustCommand*{\bfseries}{%
  \not@math@alphabet\bfseries\mathbf
  \fontseries\bfdefault\selectfont
  \boldmath
}
\makeatother

\title{The Circle Method for Quadrics over Function Fields}
\author{Johanna Mettasch}
\address{Ins\-ti\-tut für Al\-ge\-bra, Zah\-len\-the\-o\-rie und Dis\-kre\-te Ma\-the\-ma\-tik, Ins\-ti\-tut für Al\-ge\-bra\-i\-sche Ge\-o\-me\-trie\\
Leib\-niz U\-ni\-ver\-si\-tät Han\-no\-ver\\ 
Wel\-fen\-gar\-ten 1\\ 
30167 Han\-no\-ver\\
Germany}
\email{mettasch@math.uni-hannover.de}

\begin{document}
\begin{abstract}
    We use the circle method to count $\mathbb{F}_q(t)$-rational points of bounded naive height on a quadric hypersurface $X\subseteq \mathbb{P}^{n-1}$ defined over $\mathbb{F}_q$, provided that $\mathrm{char}(\mathbb{F}_q)>2$ and $n\ge 3$. Viewing these points as morphisms $\mathbb{P}^1 \to X$ of fixed degree, we obtain exact formulas for their number depending on the parity of $n$ and on the determinant of the quadratic form defining $X$, including secondary terms in some cases. 
\end{abstract}

\maketitle
\tableofcontents

\section{Introduction}
Given a non-singular form $F\in\Z[x_1,\dots,x_n]$ of degree~$d\ge 2$, we consider the counting function
\begin{align*}
    N(B)&=\#\left\{\left(x_1,\dots,x_n\right)\in\Z^n :\begin{array}{l}
   -B<x_1,\dots,x_n<B,\\
   F(x_1,\dots,x_n)=0
   \end{array} \right\},
\end{align*}
which counts the number of integral zeros of $F$ whose components are bounded by  $B>0$. Using the Hardy--Littlewood--Ramanujan circle method, Birch~\cite{birch1962} established an asymptotic formula for $N(B)$ as $B\to\infty$, provided that the number~$n$ of variables exceeds~$(d-1)2^d$.

In the case where~$d=2$, i.e. when $F$ is a quadratic form, Getz~\cite{getz2018} and Tran~\cite{tran2019} have then worked out terms of second order for~$n\ge 5$.

Over function fields, the analogue of Birch's result was given by Lee~\cite{lee2011}. We will refine this result for quadratic forms by computing the corresponding second order terms. In contrast to the setting over~$\Q$, our results yield exact formulas rather than asymptotic expansions.

Let $f\in\fq[x_1,\dots,x_n]$ be a quadratic form, where $q$ is a power of an odd prime number. Then we consider the analogous counting function
\begin{align*}
    N(P)&=\#\left\{(x_1,\dots,x_n)\in\fqt^n :\begin{array}{l}
   \deg(x_1),\dots,\deg(x_n)<P,\\
   f(x_1,\dots,x_n)=0
   \end{array} \right\}
\end{align*}
counting the number of zeros of $f$ whose components have degree smaller than $P\in\N$. 

If we additionally restrict to primitive elements and identify solutions up to multiplication by units, we obtain a related counting function that we can use to count morphisms defined over~$\fq$ of degree exactly~$P$ from~$\PP^1$ to the projective hypersurface 
\[
   X\coloneq V(f)\subseteq \PP^{n-1}.
\]
These morphisms are the $\fq$-points of the scheme
\[
   \mor_P(\PP^1,X)=\left\{g\colon \PP^1\to X: \deg(g)=P\right\}, 
\] 
and counting them is equivalent to counting $\fq(t)$-rational points on $X$ of naive height $q^P$, where the naive height is determined by the maximum degree of the homogeneous coordinates.

Using the circle method in the function field setting, together with an explicit evaluation of quadratic Gauß sums, we will arrive at the following three cases yielding different results each:
\begin{itemize}
    \item[(1)] $n$ is even and $(-1)^{\frac{n}{2}}\det(f)$ is a square in $\fq^\times$,
    \item[(2)] $n$ is even and $(-1)^{\frac{n}{2}}\det(f)$ is not a square in $\fq^\times$,
    \item[(3)] $n$ is odd. 
\end{itemize}

In the first case, our main result is the following theorem, which gives an exact formula for the cardinality of $\mor_P(\PP^1,X)(\fq)$ for~$n\ge 4$.
\begin{thm}\label{thm-quadrat:mor}
Let $n$ be even, and let $(-1)^{\frac{n}{2}}\det(f)$ be a square in $\fq^\times$. For $n=4$ we have
\begin{align*}
   \#\mor_P(\PP^1,X)(\fq)&=\frac{\left(q^2-1\right)^2}{q^2}Pq^{2P}+\frac{\left(q^2-1\right)\left(q+1\right)^2}{q^2}\,q^{2P},
\end{align*}
and for~$n\ge 6$ we have
\begin{align*}
  \#\mor_P(\PP^1,X)(\fq)&=\frac{\left(q^{\frac{n}{2}}-1\right)\left(q^{n-2}-1\right)\left(q^{n-3}-1\right)}{q^{n-2}\left(q^{\frac{n}{2}-2}-1\right)(q-1)}\, q^{P(n-2)}
    - \frac{\left(q^{n-2}-1\right)\left(q^{\frac{n}{2}}-1\right)}{q^{\frac{n}{2}}\left(q^{\frac{n}{2}-2}-1\right)}\,q^{\frac{n}{2}P}.
\end{align*}
\end{thm}

For~$n=4$, the formula in Theorem~\ref{thm-quadrat:mor} contains a term of order~$Pq^{2P}$, where the factor~$P$ plays the role of a logarithmic factor in the analogue setting over~$\Q$.

In the second case, we obtain the following result.
\begin{thm}\label{thm-keinQUADRAT:mor}
Let $n$ be even, and let $(-1)^{\frac{n}{2}}\det(f)$ not be a square in $\fq^\times$. For $n=4$ we have
\begin{align*}
   \#\mor_P(\PP^1,X)(\fq)
   =\begin{dcases}
   \frac{q^4-1}{q^2}\,q^{2P} &\textup{if }2\mid P,\\
   0 &\textup{if }2\nmid P,
   \end{dcases}
\end{align*}
and for $n\ge 6$ we have
\begin{align*}
  \#\mor_P(\PP^1,X)(\fq)&=\frac{\left(q^{\frac{n}{2}}+1\right)\left(q^{n-2}-1\right)\left(q^{n-3}-1\right)}{q^{n-2}\left(q^{\frac{n}{2}-2}+1\right)(q-1)}\, q^{P(n-2)}\\
    &\quad+(-1)^P \frac{\left(q^{n-2}-1\right)\left(q^{\frac{n}{2}}+1\right)}{q^{\frac{n}{2}}\left(q^{\frac{n}{2}-2}+1\right)}\,q^{\frac{n}{2}P}.
\end{align*}
\end{thm}

As we can see in Theorem~\ref{thm-keinQUADRAT:mor}, the formulas in the second case depend on whether $P$ is even or odd. For $n\ge 6$ this dependence only appears in the second order term, whereas for $n=4$ there is no asymptotic formula independent of the parity of $P$ since there do not even exist any morphisms $\PP^1\to X$ of degree~$P$ whenever $P$ is odd.

Geometrically, this reflects the fact that in this case the variety~$X$ is a non-split quadric surface in~$\PP^3$ as $\det(f)$ is not a square in~$\fq^\times$. Hence, the Picard group of~$X$ has rank~$1$, which implies that $X$ can only contain curves of even degree. In particular, $X$ does not admit any morphisms $\PP^1\to X$ of odd degree. By contrast, the condition $\det(f)\in (\fq^{\times})^2$ in Theorem~\ref{thm-quadrat:mor} ensures that in that case $X$ is a split quadric and thus isomorphic to 
$\PP^1\times\PP^1$ so that its Picard group has rank~$2$.

In the third case, where $n$ is odd, our result is similar to Theorem~\ref{thm-keinQUADRAT:mor}, in that the formulas for the cardinality of $\mor_P(\PP^1,X)(\fq)$ depend on the parity of~$P$.

\begin{thm}\label{thm-nUNGERADE:mor}
Let $n$ be odd. For $n=3$ we have
\begin{align*}
   \#\mor_P(\PP^1,X)(\fq)
   =\begin{dcases}
   \frac{q^2-1}{q}\, q^P &\textup{if }2\mid P,\\
   0 &\textup{if }2\nmid P,
   \end{dcases}
\end{align*}
and for $n\ge 5$ we have
\begin{align*}
   \#\mor_P(\PP^1,X)(\fq)=
   \begin{dcases}
       \frac{(q^{n-1}-1)(q^{n-2}-1)}{q^{n-2}(q-1)}\,q^{P(n-2)}&\textup{if }2\mid P,\\
       \frac{(q^{n-1}-1)(q^{n-2}-1)}{q^{n-2}(q-1)}\,q^{P(n-2)}-\frac{q^{n-1}-1}{q^{\frac{n-1}{2}}}\,q^{\frac{n-1}{2}P}&\textup{if }2\nmid P.
   \end{dcases}
\end{align*}
\end{thm}

The main terms in the formulas in Theorems~\ref{thm-quadrat:mor}, \ref{thm-keinQUADRAT:mor} and~\ref{thm-nUNGERADE:mor} are consistent with the idea of Manin's conjecture on the distribution of rational points of bounded height on Fano varieties~\cite{franke1989}. While the conjecture has originally been formulated over number fields, our results are obtained over function fields, where analogous versions of Manin's conjecture have been studied by Peyre~\cite{Peyre2012}, and more recently, by Lehmann and Tanimoto~\cite{lehmann2026}.

It would be interesting to look at the explicit formulas obtained in Theorems~\ref{thm-quadrat:mor}, \ref{thm-keinQUADRAT:mor} and~\ref{thm-nUNGERADE:mor} from a cohomological point of view. In particular, using the Grothendieck--Lefschetz trace formula, the coefficients appearing in these expressions can be interpreted in terms of the action of the Frobenius morphism on the $\ell$-adic étale cohomology with compact support of the scheme~$\mor_P(\PP^1,X)$ defined over $\fq$. This raises the question whether it is possible to give a more explicit description of the cohomology of $\mor_P(\PP^1,X)$, or of its Kontsevich compactification $\overline{\mathcal{M}}_{0,0}(X,P)$, in the spirit of the work of Bergström and Minabe~\cite{bergström2013}.

\begin{ack}
    The author wants to thank Jakob Glas for many helpful conversations. She is very grateful for his guidance and valuable feedback throughout this project. Moreover, the author would like to thank her supervisor Ulrich Derenthal for useful comments. This work was funded by the Deutsche Forschungsgemeinschaft (DFG, German Research Foundation) —-- RTG~2965 --— Project number~512730679.
\end{ack}

\section{Background on Function Fields}
Let $p$ be a prime number and $q=p^{\nu}$ for some~$\nu\in\N$. Then, up to isomorphism, $\fq$ is the unique finite field of order~$q$. Forming the field of fractions of the polynomial ring~$\fqt$, we get the function field~$\fq(t)$. The element~$t^{-1}$ now induces the absolute value~$\abs{\cdot}_\infty$ on~$\fq(t)$ by~$\abs{0}_\infty =0$ and
\[
   \abs{\frac{a}{b}}_\infty =q^{\deg(a)-\deg(b)}
\]
for all~$a,b\in\fqt\setminus\{0\}$. As~$\fq(t)$ is a field of characteristic~$\cha(\fq)=p>0$, this absolute value is non-archimedean. By completing $\fq(t)$ with respect to~$\abs{\cdot}_\infty$, we obtain the local field~$\fql$. Every element~$\alpha\in\fql$ can be written as
\[
   \alpha=\sum_{i=-\infty}^M a_it^i
\]
for some~$M\in\Z$ and coefficients $a_i\in\fq$ for~$i\le M$. If $\alpha\neq 0$, there is a maximal integer $M^\prime \le M$ such that $a_{M^\prime}\neq 0$. In this case, we define $\ord(\alpha)\coloneq M^\prime$, otherwise we define $\ord(0)\coloneq -\infty$.

Let now $\abs{\cdot}$ denote the extension of the absolute value $\abs{\cdot}_\infty$ from~$\fq(t)$ to~$\fql$. Then,
\[
   \abs{\alpha}=q^{\ord(\alpha)}
\]
for all~$\alpha\in\fql$. For~$n\in\N$ we can extend $\abs{\cdot}$ to~$\fql^n$ by setting
\[
   \abs{\boldsymbol{\alpha}}=\max\left\{\abs{\alpha_i}:i=1,\dots,n\right\}
\]
for all~$\boldsymbol{\alpha}=\left(\alpha_1,\dots,\alpha_n\right)\in\fql^n$. 

Since $\fql$ is a local field, $\fql$ is a locally compact abelian group. Hence, there exists the Haar measure, which is unique up to multiplication by a non-zero scalar. Inside the discrete valuation ring of $\fql$, there is the maximal ideal
\[
   \T\coloneq\left\{\alpha\in\fql : \abs{\alpha}<1\right\}.
\]
We can normalise the Haar measure such that $\T$ has Haar measure~$1$. For all $M\in\Z$ and $\beta\in\fql$, we then have
\begin{align}
   \int\limits_{B(\beta;M)}1 \,\mathrm{d}\alpha &=q^M,\label{gl-integralÜberBall}
\end{align}
where we denote
\[
    B(\beta;M)\coloneq\left\{\alpha\in\fql : \abs{\alpha-\beta}<q^M\right\}.
\]
We now consider the non-trivial unitary additive character
\[
   e_q\colon \fq\to\C^\times,\quad a\mapsto\exp\!\left(\frac{2\pi i \tr_{\fq /\mathbb{F}_p}(a)}{p}\right),
\]
where $\tr_{\fq / \mathbb{F}_p}\colon \fq\to\mathbb{F}_p$ denotes the trace in the field extension $\fq=\mathbb{F}_{p^\nu}\supseteq\mathbb{F}_p$. This induces the map
\[
   \psi\colon \fql\to\C^\times,
   \quad \sum_{i=-\infty}^{M} a_it^i\mapsto e_q(a_{-1})
\]
satisfying the following properties~{\cite[Lemma~5.2]{browning2021}}.

\begin{lem}\label{lem-psi}
\begin{itemize}
    \item[(i)] The map $\psi$ is trivial on $\fqt$, i.e. $\psi(\alpha)=1$ for all~$\alpha\in\fqt$.
    \item[(ii)] We have $\psi(\alpha)=\psi(\beta)$ for all $\alpha,\beta\in \fql$ with $\ord(\alpha-\beta)<-1$.
    \item[(iii)] The map $\psi$ is a non-trivial unitary additive character of $\fql$.
\end{itemize}
\end{lem}

Moreover, for all $x\in\fqt$ and $M\in\N_0$, the character $\psi$ satisfies the orthogonality relation 
\begin{align}
   \int\limits_{B(0;-M)} \psi(\alpha x)\,\mathrm{d}\alpha
   =\begin{dcases}
   \frac{1}{q^M}&\textup{if }\deg(x)<M,\\
   0&\textup{else,}
   \end{dcases}\label{gl-orthogonalitätsrelation.allgemein}
\end{align}
which in particular implies
\begin{align}
    \int\limits_{\T} \psi(\alpha x)\,\mathrm{d}\alpha=\begin{dcases}
    1&\textup{if }x=0,\\
    0&\textup{else}
    \end{dcases}\label{gl-orthogonalitätsrelation.spezialfall}
\end{align}
by setting $M=0$ (see~{\cite[Lemma~5.5 and Corollary~5.6]{browning2021}}).

\section{The Circle Method}\label{kap-kreismethode}
From now on, we always assume $\cha(\fq)=p>2$. Let $n\in\N$, and let $f\in\fq[X_1,\dots,X_n]$ be a quadratic form over $\fqt$ whose coefficients are constant, i.e. are elements of $\fq^\times$. 

For $P\in\N$ we define the counting function $N\colon \N\to\N_0$ by
\begin{align*}
   N(P)&\coloneq\#\left\{\mathbf{x}\in\fqt^n :\begin{array}{l}
   \abs{\mathbf{x}}<q^P,\\
   f(\mathbf{x})=0
   \end{array} \right\}.
\end{align*}

Since $\cha(\fq)>2$, there is a diagonal quadratic form $f^\prime\in\fq[X_1,\dots,X_n]$ such that $f$ is equivalent to~$f^\prime$. Hence, there exists an element $\mathbf{S}\in \GL_n(\fq)$ such that $ f(\mathbf{x})=f^\prime(\mathbf{S}\mathbf{x})$ for all $\mathbf{x}\in\fqt^n$. Using that the absolute value $\abs{\cdot}$ is non-archimedean and thus satisfies the ultrametric inequality, we can show that for all $P\in\N$ the value $N(P)$ does not change if we replace $f$ by $f^\prime$. Therefore, the counting function is invariant under diagonalisation, and we may without loss of generality assume that $f$ is already diagonalised, i.e. that there are $a_1,\dots,a_n\in\fq^\times$ such that
\[
   f(X_1,\dots,X_n)=a_1X_1^2+\dots +a_nX_n^2.
\]

We further define
\begin{align*}
   \tilde{N}(P)
    &\coloneq\frac{1}{q-1}\cdot \#\left\{(x_1,\dots,x_n)\in\fqt^n :\begin{array}{l}
   \abs{(x_1,\dots,x_n)}<q^P,\\
   f(x_1,\dots,x_n)=0,\\
   \gcd(x_1,\dots,x_n)=1
   \end{array} \right\}
\end{align*}
to count all primitive zeros of $f$ in $\fqt^n$ of degree less than $P$ up to multiplication by a unit, i.e. an element of $\fq^\times$. 

Let $\mu$ denote the function field analogue of the Möbius function. Applying Möbius inversion then yields
\begin{align*}
    \tilde{N}(P)
    &=\frac{1}{q-1}\sum_{\substack{r\in\fqt\\ \abs{r}\le q^P\\ r\textup{ monic}}} \mu(r)\left(N(P-\deg(r))-1\right).
\end{align*}
Since we have 
\begin{align*}
    \sum_{\substack{r\in\fqt\\ \deg(r)=\rho\\ r\textup{ monic}}}\mu(r)
    =\begin{dcases}
        1&\textup{if }\rho=0,\\
        -q&\textup{if }\rho=1,\\
        0&\textup{if }\rho\ge 2
    \end{dcases}
\end{align*}
for all $\rho\in\N_0$ (see~{\cite[Exercise~12]{rosen2002}}), it follows
\begin{align}
    \tilde{N}(P)
    &=\frac{N(P)-q N(P-1)}{q-1}+1\label{gl-tildeN(P)}.
\end{align}

Let $X\subseteq \PP^{n-1}$ be the variety defined by $f$. A morphism $\PP^1\to X$ of degree $P$ is given by a primitive $n$-tuple $(g_0,\dots,g_{n-1})$ where $g_0,\dots,g_{n-1}$ are homogeneous binary forms of degree exactly $P$ satisfying $f(g_0,\dots,g_{n-1})=0$. Such morphisms correspond to $\fq(t)$-rational points on $X$ of naive height $q^P$.

As $\tilde N(P+1)$ counts the primitive solutions of degree at most $P$, it follows from~\eqref{gl-tildeN(P)} that the number of morphisms from $\PP^1$ to $X$ of degree exactly $P$ amounts to
\begin{align}
    \tilde{N}(P+1)-\tilde{N}(P)
    &=\frac{1}{q-1}\left(N(P+1)-(q+1)N(P)+qN(P-1)\right).\label{gl-mor_P}
\end{align}
These morphisms are precisely the $\fq$-points of the scheme $\mor_P(\PP^1,X)$, and their number is hence given by $\#\mor_P(\PP^1,X)(\fq)$. To determine this cardinality and thus count $\fq(t)$-rational points on $X$ of height exactly $q^P$, it therefore suffices to compute $N(P)$ and insert into~\eqref{gl-mor_P}. 

Let now $P\in\N$ be fixed. We define the exponential sums
\begin{align*}
    T(\alpha)&\coloneq \sum_{\substack{x\in\fqt\\\abs{x}<q^P}} \psi\!\left(\alpha x^2\right),\\
   S(\alpha)&\coloneq \sum_{\substack{\mathbf{x}\in\fqt^n\\\abs{\mathbf{x}}<q^P}} \psi\!\left(\alpha f(\mathbf{x})\right)
\end{align*}
for $\alpha\in\fql$. Then, we have
\begin{align}
   S(\alpha)&=\prod_{i=1}^n T\!\left(a_i\alpha\right),\label{gl-S(alpha)=produkt}
\end{align}
and since $\psi$ satisfies the orthogonality relation in~\eqref{gl-orthogonalitätsrelation.spezialfall}, we can write
\[
   N(P)=\int\limits_{\T} S(\alpha)\,\mathrm{d}\alpha.
\]

From the function field analogue of Dirichlet's approximation theorem~{\cite[Lemma~5.7]{browning2021}} and the ultrametric inequality, that the absolute value $\abs{\cdot}$ satisfies, it follows that the set $\T$ can be written as the disjoint union
\[
    \T=\bigsqcup_{\substack{r\in\fqt\\
   \abs{r}\le q^P\\
   r\textup{ monic}}}
   \bigsqcup_{\substack{ a\in\fqt\\
   \abs{a} < \abs{r}\\
   \gcd (a,r)=1  }}
   \left\{ \alpha\in\T : \abs{r\alpha-a}<\frac{1}{q^P}  \right\}
\]
(see~{\cite[Lemma~5.8]{browning2021}}). Thus, we obtain
\begin{align}
    N(P)
    &= \sum_{\substack{r\in\fqt\\
    \abs{r}\le q^P\\
    r\textup{ monic}}}
    \sum_{\substack{ a\in\fqt\\
    \abs{a} < \abs{r}\\
    \gcd (a,r)=1}}
    \int\limits_{\left\{\theta\in\T:\abs{\theta}<\frac{1}{\abs{r}q^{P}}\right\}} 
    S\!\left(\frac{a}{r}+\theta\right)\,\mathrm{d}\theta.\label{gl-N(P)=SumSumInt}
\end{align}
Our goal in this section is to express $N(P)$ in terms of exponential sums of the form
\[
   S_{a,r}(f)\coloneq \sum_{\substack{\mathbf{b}\in\fqt^n \\ \abs{\mathbf{b}}<\abs{r}}} \psi\!\left(\frac{af(\mathbf{b})}{r}\right)
\]
for coprime $a,r\in\fqt$ with $r\neq 0$. Using quadratic Gauß sums, we will be able to compute these sums, which we will do in section~\ref{kap-quadratischeGaußsummen}. 

We notice that by defining
\[
    S_{a,r}\coloneq\sum_{\substack{x\in\fqt \\ \abs{x}<\abs{r}}} \psi\!\left(\frac{ax^2}{r}\right),
\]
we can factorise $S_{a,r}(f)$ into
\begin{align}
    S_{a,r}(f)&=\prod_{i=1}^n S_{a_ia,r},\label{gl-S_a,r(f)=produkt}
\end{align}
in analogy with~\eqref{gl-S(alpha)=produkt}, where we write $S(\alpha)$ as the product of $n$ exponential sums corresponding to the monomials of $f$.

To obtain an expression of $N(P)$ in terms of exponential sums $S_{a,r}(f)$, we begin by examining the integrand of the integrals appearing in~\eqref{gl-N(P)=SumSumInt}.

\begin{lem}\label{lem-S(a/r+theta)}
Let $a,r\in\fqt$ be coprime with $\abs{a}<\abs{r}\le q^P$, and let $\theta\in\T$ such that $\abs{\theta}<\frac{1}{\abs{r}q^{P}}$. Then, we have
\[
    S\!\left(\frac{a}{r}+\theta\right) =\frac{S_{a,r}(f)\,S(\theta)}{\abs{r}^n}.
\]
\end{lem}

\begin{proof}
First, we show
\begin{align}
   T\!\left(\frac{a}{r}+\theta\right)&=\frac{S_{a,r}\,T(\theta)}{\abs{r}}.\label{gl-T(a/r+theta)=S_arT(theta)}
\end{align}
For all $x\in\fqt$ with $\abs{x}<q^P$, there are unique $b,y\in\fqt$ with $\abs{b}<\abs{r}$ and $\abs{y}<\frac{q^P}{\abs{r}}$ such that $x=b+ry$. Hence, we can write
\begin{align*}
   T\!\left(\frac{a}{r}+\theta\right)
   &=\sum_{\substack{x\in\fqt \\ \abs{x}<q^P}} \psi\!\left(\left(\frac{a}{r}+\theta\right)x^2\right)\\
   &=\sum_{\substack{b\in\fqt \\ \abs{b}<\abs{r}}} \sum_{\substack{y\in\fqt \\ \abs{y}<\frac{q^P}{\abs{r}}}} \psi\!\left(\left(\frac{a}{r}+\theta\right)\left(b+ry\right)^2\right)\\
   &=\sum_{\substack{b\in\fqt \\ \abs{b}<\abs{r}}} \sum_{\substack{y\in\fqt \\ \abs{y}<\frac{q^P}{\abs{r}}}} \psi\!\left(\frac{a\left(b+ry\right)^2}{r}\right) \psi\!\left(\theta\left(b+ry\right)^2\right).
\end{align*}
According to Lemma~\ref{lem-psi}, we have
\begin{align*}
   \psi\!\left(\frac{a\left(b+ry\right)^2}{r}\right)
   &=\psi\!\left(\frac{ab^2}{r}\right)\psi\!\left(2aby+ ary^2\right)
   =\psi\!\left(\frac{ab^2}{r}\right)
\end{align*}
for all $b\in\fqt$, and thus
\begin{align*}
   T\!\left(\frac{a}{r}+\theta\right)
   &=\sum_{\substack{b\in\fqt \\ \abs{b}<\abs{r}}} \psi\!\left(\frac{ab^2}{r}\right) \sum_{\substack{y\in\fqt \\ \abs{y}<\frac{q^P}{\abs{r}}}} \psi\!\left(\theta\left(b+ry\right)^2\right).
\end{align*}
Suppose now that $b,y\in\fqt$ are such that $\abs{b}<\abs{r}$ and $\abs{y}<\frac{q^P}{\abs{r}}$. In particular, we then have $\abs{b}\le \frac{\abs{r}}{q}$ and hence
\begin{align*}
   \abs{\theta b^2}
   &<\frac{1}{\abs{r}q^{P}}\left(\frac{\abs{r}}{q}\right)^2
   =\frac{\abs{r}}{q^{P+2}}
   \le \frac{q^{P}}{q^{P+2}}=\frac{1}{q^2}
\intertext{as well as}
   \abs{2\theta bry}
   &<\abs{2}\cdot \frac{1}{\abs{r}q^P}\cdot\frac{\abs{r}}{q}\cdot\abs{r}\cdot\frac{q^P}{\abs{r}}
   =\frac{1}{q}
\end{align*}
since we assumed $\abs{\theta}<\frac{1}{\abs{r}q^P}$ and $\abs{r}\le q^P$. Consequently, we get
\[
   \abs{\theta(b+ry)^2-\theta r^2y^2}
   =\abs{\theta b^2+2\theta bry}
   \le \max\{\abs{\theta b^2},\abs{2\theta bry}\}
   <\frac{1}{q},
\]
which implies
\[
    \ord\!\left(\theta(b+ry)^2-\theta r^2y^2\right)<-1.
\]
From Lemma~\ref{lem-psi}, it therefore follows that
\[
   \psi\!\left(\theta(b+ry)^2\right)=\psi\!\left(\theta r^2y^2\right),
\]
and thus
\begin{align*}
   T\!\left(\frac{a}{r}+\theta\right)
   &=\sum_{\substack{b\in\fqt \\ \abs{b}<\abs{r}}} \psi\!\left(\frac{ab^2}{r}\right) \sum_{\substack{y\in\fqt \\ \abs{y}<\frac{q^P}{\abs{r}}}} \psi\!\left(\theta r^2y^2\right)\\
   &=S_{a,r} \sum_{\substack{y\in\fqt \\ \abs{y}<\frac{q^P}{\abs{r}}}} \psi\!\left(\theta r^2y^2\right).
\end{align*}
For $a=0$ this yields
\begin{align*}
    T(\theta)
   &=\sum_{\substack{b\in\fqt \\ \abs{b}<\abs{r}}} \psi(0)
    \sum_{\substack{y\in\fqt \\ \abs{y}<\frac{q^P}{\abs{r}}}} \psi\!\left(\theta r^2y^2\right)
    =\abs{r}\sum_{\substack{y\in\fqt \\ \abs{y}<\frac{q^P}{\abs{r}}}} \psi\!\left(\theta r^2y^2\right),
\end{align*}
and we then get
\begin{align*}
    T\!\left(\frac{a}{r}+\theta\right)
    &=S_{a,r} \sum_{\substack{y\in\fqt \\ \abs{y}<\frac{q^P}{\abs{r}}}} \psi\!\left(\theta r^2y^2\right)
    =S_{a,r}\,\frac{T(\theta)}{\abs{r}}
\end{align*}
and hence~\eqref{gl-T(a/r+theta)=S_arT(theta)}. From~\eqref{gl-S(alpha)=produkt} and~\eqref{gl-T(a/r+theta)=S_arT(theta)}, we now obtain
\begin{align*}
   S\!\left(\frac{a}{r}+\theta\right)
   &=\prod_{i=1}^n T\!\left( \frac{a_i a}{r}+a_i\theta\right)
   = \prod_{i=1}^n \frac{S_{a_i a,r} \,T(a_i\theta)}{\abs{r}}.
\end{align*}
With that, it finally follows
\begin{align*}
   S\!\left(\frac{a}{r}+\theta\right)
   &=\frac{1}{\abs{r}^n}\left(\prod\limits_{i=1}^n S_{a_i a,r}\right)\left(\prod\limits_{i=1}^n T(a_i\theta)\right)
   = \frac{S_{a,r}(f)\,S(\theta)}{\abs{r}^n}
\end{align*}
from~\eqref{gl-S_a,r(f)=produkt} and~\eqref{gl-S(alpha)=produkt}.
\end{proof}

Applying Lemma~\ref{lem-S(a/r+theta)} to the integrand in~\eqref{gl-N(P)=SumSumInt} yields
\begin{align}
    N(P)
    &=\sum_{\substack{r\in\fqt\\
   \abs{r}\le q^P\\
   r\textup{ monic}}}
   \frac{S_{r}(f)}{\abs{r}^n}
   \int\limits_{\left\{\theta\in\T:\abs{\theta}<\frac{1}{\abs{r}q^{P}}\right\}}  S(\theta)\,\mathrm{d}\theta,\label{gl-N(P)=SumS_r(f)Int}
\end{align}
where we denote
\begin{align}
S_r(f)&\coloneq\sum_{\substack{a\in\fqt \\ \abs{a}<\abs{r}\\ \gcd(a,r)=1}} S_{a,r}(f)\label{gl-S_r(f)Definition}
\end{align}
for monic $r\in\fqt$. In the following lemma we show that the integral
\[
    I_r\coloneq\int\limits_{\left\{\theta\in\T:\abs{\theta}<\frac{1}{\abs{r}q^{P}}\right\}}  S(\theta)\,\mathrm{d}\theta,
\]
that appears in~\eqref{gl-N(P)=SumS_r(f)Int} for monic $r\in\fqt$ with $\abs{r}\le q^P$, can be rewritten in terms of sums of the form in~\eqref{gl-S_r(f)Definition}.

\begin{lem}\label{lem-Integral=quasiSumGaußsummen}
Let $r\in\fqt$ be monic such that $\rho\coloneq\deg(r)\le P$. We have
\begin{align*}
    I_r =\begin{dcases}
        \frac{\abs{r}^n q^{n+1}}{q^{2P}}\sum_{k=0}^{P-\rho-1}q^{nk}\,S_{t^{P-\rho-k-1}}(f)&\textup{if }\rho\le P-1,\\
        q^{P(n-2)}&\textup{if }\rho=P.
    \end{dcases}
\end{align*}
\end{lem}

\begin{proof}
First, we assume $\rho=P$. Let $\mathbf{x}\in\fqt^n$ with $\abs{\mathbf{x}}<q^P$. According to~\eqref{gl-orthogonalitätsrelation.allgemein}, we have
\begin{align}
    \int\limits_{\left\{\theta\in\T : \abs{\theta}<\frac{1}{q^{\rho+P}}\right\}} \psi(\theta f(\mathbf{x}))\,\mathrm{d}\theta
    =\begin{dcases}
        \frac{1}{q^{2P}}&\textup{if }\abs{f(\mathbf{x})}<q^{2P},\\
        0&\textup{else.}
    \end{dcases}\label{gl-IntegralFalls:rho=P}
\end{align}
Since $f$ is a quadratic form with coefficients in $\fq^\times$, it holds
\begin{align}
    \abs{f(\mathbf{x})}
    &\le \abs{\mathbf{x}}^2
    \le \left(q^{P-1}\right)^2
    <q^{2P},\label{gl-Betragf(x)<=q2P-2}
\end{align}
and hence, the integral in~\eqref{gl-IntegralFalls:rho=P} has the value $\frac{1}{q^{2P}}$. Summing over all $\mathbf{x}\in\fqt^n$ with $\abs{\mathbf{x}}<q^P$ therefore yields
\begin{align*}
    I_r
    &=\sum_{\substack{\mathbf{x}\in\fqt^n\\ \abs{\mathbf{x}}<q^P }} \int\limits_{\left\{\theta\in\T : \abs{\theta}<\frac{1}{q^{\rho+P}}\right\}} \psi(\theta f(\mathbf{x}))\,\mathrm{d}\theta =\sum_{\substack{\mathbf{x}\in\fqt^n\\ \abs{\mathbf{x}}<q^P }} \frac{1}{q^{2P}}
    =q^{nP}\cdot \frac{1}{q^{2P}}=q^{P(n-2)}.
\end{align*}

Now we assume $\rho\le P-1$. We set $s\coloneq t^{-1}$. Let $\theta\in\T$ with $\abs{\theta}<\frac{1}{q^{\rho+P}}$. Then, for all $i\in\Z$ with $i<-(\rho+P)$, there is a $b_i\in\fq$ such that
\[
    \theta=\sum_{i=-\infty}^{-\rho-P-1} b_i t^i=\sum_{i=\rho+P+1}^\infty b_{-i}s^i.
\]
We can write $\theta$ as the sum of
\[
    \theta_1\coloneq \sum_{i=\rho+P+1}^{2P-1} b_{-i} s^i
    =s^{\rho+P+1}\sum_{i=0}^{P-\rho-2} b_{-\rho-P-1-i} s^i 
\]
and
\[
    \theta_2\coloneq \sum_{i=-\infty}^{-2P} b_{i} t^i
    \in B(0;-2P+1).
\]
Since for all $\mathbf{x}\in\fqt^n$ with $\abs{\mathbf{x}}<q^P$, we have $f(\mathbf{x})\le q^{2P-2}$ according to~\eqref{gl-Betragf(x)<=q2P-2}, it then holds
\[
    \abs{\theta_2f(\mathbf{x})}\le q^{-2P}q^{2P-2}=\frac{1}{q^2}.
\]
From Lemma~\ref{lem-psi}, it thus follows
\[
    \psi(\theta f(\mathbf{x}))=\psi(\theta_1 f(\mathbf{x}))\psi(\theta_2 f(\mathbf{x}))=\psi(\theta_1 f(\mathbf{x}))
\]
for all $\mathbf{x}\in\fqt^n$ with $\abs{\mathbf{x}}<q^P$, i.e. $\psi$ only depends on $\theta_1$, but not on $\theta_2$. Consequently, we have $S(\theta)=S(\theta_1)$. By viewing $\theta_1$ as an element of the quotient $s^{\rho+P+1}R_\rho$, where $R_\rho\coloneq \mathbb{F}_q[s]/\!\left(s^{P-\rho-1}\right)$, we can decompose the set we want to integrate over into
\[
    \left\{\theta\in\T : \abs{\theta}<\frac{1}{q^{\rho+P}}\right\}
    =s^{\rho+P+1}R_\rho + B(0;-2P+1),
\]
where the integrand $S(\theta)$ only depends on the corresponding element in $s^{\rho+P+1}R_\rho$ and is constant for such a fixed element of $s^{\rho+P+1}R_\rho$. Since the set $B(0;-2P+1)$ has Haar measure $q^{-2P+1}$ according to~\eqref{gl-integralÜberBall}, we get
\[
    I_r
    =q^{-2P+1}\sum_{\theta\in s^{\rho+P+1}R_\rho} S(\theta)
    =q^{-2P+1}\sum_{\theta\in R_\rho} S\!\left(s^{\rho+P+1}\theta\right).
\]

We introduce the notation $\abs{\cdot}_s$ for the absolute value when we measure the degree or the order of an element with respect to the variable $s$ instead of $t$. Let $x\in\fqt$ with $\abs{x}<q^P$. Then, there are coefficients $x_0,\dots,x_{P-1}\in\fq$ such that
\[
    x=\sum_{i=0}^{P-1}x_it^i=\sum_{i=-(P-1)}^0 x_{-i}s^i=s^{1-P}\sum_{i=0}^{P-1}x_{P-1-i} s^i.
\]
Hence, every $\mathbf{x}\in\fqt^n$ with $\abs{\mathbf{x}}<q^P$ can be written as $\mathbf{x}=s^{1-P}\mathbf{x^\prime}$ for some $\mathbf{x^\prime}\in\mathbb{F}_q[s]^n$ with $\abs{\mathbf{x^\prime}}_s<q^P$, and we then have
\[
    f(\mathbf{x})=f\!\left(s^{1-P}\mathbf{x^\prime}\right)=s^{2-2P}f(\mathbf{x^\prime}).
\]
From that, we obtain
\begin{align*}
    I_r&=q^{-2P+1}\sum_{\theta\in R_\rho} \sum_{\substack{\mathbf{x}\in\fqt^n\\ \abs{\mathbf{x}}<q^P}}\psi\!\left(s^{\rho+P+1}\theta f(\mathbf{x})\right)\\
    &=q^{-2P+1}\sum_{\theta\in R_\rho} \sum_{\substack{\mathbf{x}\in\mathbb{F}_q[s]^n\\ \abs{\mathbf{x}}_s <q^P}}\psi\!\left(\frac{\theta f(\mathbf{x})}{s^{P-\rho-3}}\right).
\end{align*}
For every $x\in\fqt$ with $\abs{x}_s<q^P$, we can write $x=y+s^{P-\rho-1} z$ for unique $y,z\in\fqt$ with $\abs{y}_s< q^{P-\rho-1}$ and $\abs{z}_s<\frac{q^P}{q^{P-\rho-1}}=q^{\rho+1}$. Then, we have
\begin{align*}
    \frac{\theta x^2}{s^{P-\rho-3}}-\frac{\theta y^2}{s^{P-\rho-3}}
    &=\frac{2\theta ys^{P-\rho-1}z +\theta s^{2(P-\rho-1)}z^2}{s^{P-\rho-3}} =2\theta yzs^{2}+\theta z^2s^{P-\rho+1}
\end{align*}
for all $\theta\in R_\rho$. As we assumed $\rho\le P-1$, the monomial $s^i$ only appears for $i>1$. Since the character $\psi$ only depends on the coefficient of $t^{-1}=s$, it then follows 
\[
    \psi\!\left(\frac{\theta x^2}{s^{P-\rho-3}}\right)=\psi\!\left(\frac{\theta y^2}{s^{P-\rho-3}}\right)
\]
from Lemma~\ref{lem-psi}, and hence
\[
    \psi\!\left(\frac{\theta f(\mathbf{x})}{s^{P-\rho-3}}\right)=\psi\!\left(\frac{\theta f(\mathbf{y})}{s^{P-\rho-3}}\right)
\]
for all $\mathbf{x}\in\fqt^n$ where $\mathbf{x}=\mathbf{y}+s^{P-\rho-1}\mathbf{z}$ for unique $\mathbf{y},\mathbf{z}\in\fqt^n$ with $\abs{\mathbf{y}}_s<q^{P-\rho-1}$ and $\abs{\mathbf{z}}_s<q^{\rho+1}$. This yields
\begin{align*}
    I_r &=q^{-2P+1}\sum_{\theta\in R_\rho} \sum_{\substack{\mathbf{y}\in\mathbb{F}_q[s]^n\\ \abs{\mathbf{y}}_s <q^{P-\rho-1}}}\sum_{\substack{\mathbf{z}\in\mathbb{F}_q[s]^n\\ \abs{\mathbf{z}}_s <q^{\rho+1}}} \psi\!\left(\frac{\theta f(\mathbf{y})}{s^{P-\rho-3}}\right)\\
     &=\frac{q^{n(\rho+1)+1}}{q^{2P}}\sum_{\theta\in R_\rho} \sum_{\substack{\mathbf{x}\in\mathbb{F}_q[s]^n\\ \abs{\mathbf{x}}_s <q^{P-\rho-1}}}\psi\!\left(\frac{\theta f(\mathbf{x})}{s^{P-\rho-3}}\right)
\end{align*}
because the number of $\mathbf{z}\in\fqt^n$ with $\abs{\mathbf{z}}_s<q^{\rho+1}$ amounts to $q^{n(\rho+1)}$. We now define $\psi_s$ via
\[
    \psi_s(\alpha)\coloneq \psi(s^2\alpha)
\]
for all $\alpha\in\fql$. Since $\psi$ only depends on the coefficient of $t^{-1}=s$, the character $\psi_s$ only depends on the coefficient of $s^{-1}$. We can then write
\begin{align}
    I_r &=\frac{q^{n(\rho+1)+1}}{q^{2P}}\sum_{\theta\in R_\rho} \sum_{\substack{\mathbf{x}\in\mathbb{F}_q[s]^n\\ \abs{\mathbf{x}}_s <q^{P-\rho-1}}}\psi_s\!\left(\frac{\theta f(\mathbf{x})}{s^{P-\rho-1}}\right)\notag\\
    &=\frac{\abs{r}^n q^{n+1}}{q^{2P}}\sum_{\substack{\theta\in \fqt\\ \abs{\theta}<q^{P-\rho-1}}} \sum_{\substack{\mathbf{x}\in\mathbb{F}_q[t]^n\\ \abs{\mathbf{x}} <q^{P-\rho-1}}}\psi\!\left(\frac{\theta f(\mathbf{x})}{t^{P-\rho-1}}\right),\label{gl-I_rLetzterSchritt}
\end{align}
where we just replace $s$ by $t$, and therefore $\psi_s$ by $\psi$. Let now $\theta\in\fqt$ such that $\abs{\theta}<q^{P-\rho-1}$. Then, we can find a maximal integer $k\in\{0,\dots,P-\rho-1\}$ with $t^k\mid \theta$ and hence, we may write $\theta=t^k \,\theta^\prime$ for some $\theta^\prime\in\fqt$ such that $t$ and $\theta^\prime$ are coprime. Here, we notice that we have $k=P-\rho-1$ if and only if $\theta=0$. Since for all $\mathbf{x}\in\fqt^n$ with $\abs{\mathbf{x}}<q^P$, there exist unique $\mathbf{y},\mathbf{z}\in\fqt$ with $\abs{\mathbf{y}}<q^{P-\rho-k-1}$ and $\abs{\mathbf{z}}<q^k$ such that $\mathbf{x}=\mathbf{y}+t^{P-\rho-k-1}\mathbf{z}$, we then have
\[
    f(\mathbf{x})=f\!\left(\mathbf{y}+t^{P-\rho-k-1}\mathbf{z}\right)=f(\mathbf{y})+t^{P-\rho-k-1}v
\]
for some $v\in\fqt$, and thus
\[
    \frac{\theta f(\mathbf{x})}{t^{P-\rho-1}}=\frac{t^k\theta^\prime \left(f(\mathbf{y})+t^{P-\rho-k-1}v\right)}{t^{P-\rho-1}}
    =\frac{\theta^\prime f(\mathbf{y})}{t^{P-\rho-k-1}}+\theta^\prime v.
\]
From Lemma~\ref{lem-psi}, it therefore follows
\begin{align*}
    \sum_{\substack{\mathbf{x}\in\mathbb{F}_q[t]^n\\ \abs{\mathbf{x}} <q^{P-\rho-1}}}\psi\!\left(\frac{\theta f(\mathbf{x})}{t^{P-\rho-1}}\right)
    &= \sum_{\substack{\mathbf{y}\in\mathbb{F}_q[t]^n\\ \abs{\mathbf{y}} <q^{P-\rho-k-1}}}\sum_{\substack{\mathbf{z}\in\mathbb{F}_q[t]^n\\ \abs{\mathbf{z}} <q^k}} \psi\!\left(\frac{\theta^\prime f(\mathbf{y})}{t^{P-\rho-k-1}}\right)\\
    &= \sum_{\substack{\mathbf{y}\in\mathbb{F}_q[t]^n\\ \abs{\mathbf{y}} <q^{P-\rho-k-1}}} \left(q^k\right)^n \,\psi\!\left(\frac{\theta^\prime f(\mathbf{y})}{t^{P-\rho-k-1}}\right)\\
    &= q^{nk} \,S_{\theta^\prime, t^{P-\rho-k-1}}(f)
\end{align*}
since $\theta^\prime$ and $t^{P-\rho-k-1}$ are coprime. By inserting into~\eqref{gl-I_rLetzterSchritt}, we then obtain
\begin{align*}
    I_r &=\frac{\abs{r}^n q^{n+1}}{q^{2P}}\sum_{k=0}^{P-\rho-1}\sum_{\substack{\theta\in \fqt\\ \abs{\theta}<q^{P-\rho-k-1}\\ \gcd(\theta,t)=1}} q^{nk}\, S_{\theta, t^{P-\rho-k-1}}(f).
\end{align*}
\end{proof}

\section{Quadratic Gauß Sums}\label{kap-quadratischeGaußsummen}
Let $r\in\fqt$ be monic. We define the \emph{quadratic Gauß sum} $\tau_r$ by
\[
  \tau_r \coloneq \sum_{\substack{x\in\fqt\\ \abs{x}<\abs{r}}} \psi\!\left(\frac{x^2}{r}\right).
\]
Since the character $\psi$ induces a well-defined character $\psi_r$ on the ring $\fqt/(r)$ by
\[
  \psi_r\colon\fqt/(r)\to\C^\times, \quad [x]\mapsto\psi\!\left(\frac{x}{r}\right),
\]
we can write
\[
  \tau_r=\sum_{x\in\fqt/(r)}\psi_r\!\left(x^2\right).
\]
For a monic irreducible element $\varpi\in\fqt$ we also have
\begin{align}
  \tau_\varpi
  &=\sum_{x\in\fqt/(\varpi)} \legendre{x}{\varpi}\psi_\varpi(x)\label{gl-tau_varpiMitLegendresymbol}
\end{align}
according to~{\cite[Equation~(1.1.4)]{berndt1998}} because then $\fqt/(\varpi)$ is a finite field. Here, $\legendre{\cdot}{\varpi}$ denotes the function field analogue of the Legendre symbol. As $\fqt/(\varpi)$ contains
\[
    \abs{\varpi}= q^{\deg(\varpi)}=p^{\deg(\varpi)\nu}
\]
elements and is therefore isomorphic to $\mathbb{F}_{p^{\deg(\varpi)\nu}}$, the expression of the Gauß sum $\tau_\varpi$ in~\eqref{gl-tau_varpiMitLegendresymbol} allows us to apply~{\cite[Theorem~11.5.4]{berndt1998}}, which yields
\[
  \tau_\varpi= -i_p^{\,\deg(\varpi)\nu}\abs{\varpi}^{\frac{1}{2}},
\]
where $i_p$ is given by
\begin{align*}
     i_p=\begin{dcases}
     -1&\textup{if }p\equiv 1\mod 4,\\
     -i&\textup{if }p\equiv 3\mod 4.
     \end{dcases}
  \end{align*}
The following result~{\cite[Lemma~2.3]{vishe2023}} then shows how to evaluate $\tau_{\varpi^k}$ for $k\ge 2$.

\begin{lem}\label{lem-tau_varpikPotenz}
Let $\varpi\in\fqt$ be monic and irreducible, and let $k\in\N$. Then, we have
\begin{align*}
    \tau_{\varpi^k}
    =\begin{dcases}
        \abs{\varpi}^{\frac{k}{2}}&\textup{if }2\mid k,\\
        -i_p^{\,\deg(\varpi)\nu}\abs{\varpi}^{\frac{k}{2}}&\textup{if }2\nmid k.
    \end{dcases}
\end{align*}
\end{lem}

For $k\in\N$ and $a,\varpi\in\fqt$ such that $\varpi$ is monic, irreducible and coprime to $a$, we can express sums of the form $S_{a,\varpi^k}$ in terms of $\tau_{\varpi^k}$ as we will show in the next lemma. To this end, we make use of the function field analogue of the Jacobi symbol, which generalises the function field analogue of the Legendre symbol, analogously to the definition over the integers.

\begin{lem}\label{lem-LegendresymbolRausziehen}
Let $\varpi\in\fqt$ be a monic irreducible element, and let $a\in\fqt$ be coprime to $\varpi$. For all $k\in\N$ we have
\[
  S_{a,\varpi^k} =\legendre{a}{\varpi^k}\tau_{\varpi^k}.
\]
\end{lem}

\begin{proof}
First, we assume $k=1$. Then, $\fqt/(\varpi)$ is isomorphic to the finite field $\mathbb{F}_{p^{\nu \deg(\varpi)}}$, and the quadratic Gauß sum~$\tau_\varpi$ can be expressed as in~\eqref{gl-tau_varpiMitLegendresymbol}. Hence, the identity we want to show directly follows from~{\cite[Theorem~1.1.3]{berndt1998}}.

If $k\ge 2$, we can proceed similarly to the proof of Lemma~\ref{lem-tau_varpikPotenz} (see~{\cite[Lemma~2.3]{vishe2023}}) and deduce
\begin{align*}
   S_{a,\varpi^k}=\begin{dcases}
   \abs{\varpi}^{\frac{k}{2}}&\textup{if }2\mid k,\\
   \abs{\varpi}^{\frac{k-1}{2}}S_{a,\varpi}&\textup{if }2\nmid k.
   \end{dcases}
\end{align*}
We then immediately get
\begin{align*}
    S_{a,\varpi^k}
    &=\abs{\varpi}^{\frac{k}{2}}
    =\tau_{\varpi^k}
    =\legendre{a}{\varpi^k}\tau_{\varpi^k}
\end{align*}
if $k$ is even. For the case where $k\ge 3$ is odd, we additionally use $S_{a,\varpi}=\legendre{a}{\varpi}\tau_\varpi$ from the case $k=1$, which yields
\begin{align*}
   S_{a,\varpi^k}
   &=\abs{\varpi}^{\frac{k-1}{2}}S_{a,\varpi}
   =\abs{\varpi}^{\frac{k-1}{2}}\legendre{a}{\varpi}\tau_{\varpi}.
\end{align*}
From the second equation in the proof of~~{\cite[Lemma~2.3]{vishe2023}}, we then get the desired identity.
\end{proof}

We now want to compute the sum
\begin{align*}
    S_r(f)
    &=\sum_{\substack{a\in\fqt \\ \abs{a}<\abs{r}\\ \gcd(a,r)=1}} S_{a,r}(f)
    =\sum_{\substack{a\in\fqt \\ \abs{a}<\abs{r}\\ \gcd(a,r)=1}} \prod_{i=1}^n S_{a_ia,r}
\end{align*}
for monic $r\in\fqt$, introduced in~\eqref{gl-S_r(f)Definition}. For this, we use the following lemma, which is the analogue of~{\cite[Lemma~2.13]{browning2021}} for function fields.

\begin{lem}\label{lem-S_r(f)ProduktTeilerfremd}
    Let $r_1,r_2\in\fqt$ be monic and coprime. Then, we have
    \[
        S_{r_1r_2}(f)=S_{r_1}(f)\,S_{r_2}(f).
    \]
\end{lem}

For every monic $r\in\fqt$ with $\deg(r)\ge 1$, we can uniquely write
\[
    r=\varpi_1^{k_1}\cdots \varpi_m^{k_m}
\]
for pairwise distinct monic irreducible elements $\varpi_1,\dots,\varpi_m\in\fqt$ and $k_1,\dots,k_m, m\in\N$. From Lemma~\ref{lem-S_r(f)ProduktTeilerfremd}, it follows that $S_r(f)$ then admits a factorisation
\begin{align}
    S_r(f)
    &=\prod_{i=1}^m S_{\varpi_i^{k_i}}(f).\label{gl-faktorisierungS_r(f)}
\end{align}
To compute $S_r(f)$, it therefore suffices to only look at sums of the form $S_{\varpi^k}(f)$ for a monic irreducible element $\varpi\in\fqt$ and $k\in\N$. For these exponential sums we have the following result~{\cite[Lemma~2.4]{vishe2023}}.

\begin{lem}\label{lem-S_varpiPotenz}
Let $\varpi\in\fqt$ be a monic irreducible element, and let $k\in\N$. If $n$ is even, we have
\[
    S_{\varpi^k}(f)=\legendre{(-1)^{\frac{n}{2}} a_1\cdots a_n}{\varpi^k}\varphi(\varpi^k) \abs{\varpi^k}^{\frac{n}{2}},
\]
and if $n$ is odd, we have
\[
    S_{\varpi^k}(f)=\begin{dcases}
        \varphi(\varpi^k) \abs{\varpi^k}^{\frac{n}{2}}&\textup{if }2\mid k,\\
        0&\textup{if }2\nmid k,
    \end{dcases}
\]
where $\varphi$ denotes the function field analogue of Euler's totient function.
\end{lem}

\begin{proof}
Let $a\in\fqt$ with $\abs{a}<\abs{\varpi}^k$ and $\gcd(a,\varpi^k)=1$. Applying~\eqref{gl-S_a,r(f)=produkt} and Lemma~\ref{lem-LegendresymbolRausziehen}, we get
 \[
  S_{a,\varpi^k}(f)=\prod_{i=1}^n S_{a_ia,\varpi^k}
  =\prod_{i=1}^n \legendre{a_ia}{\varpi^k}\tau_{\varpi^k}
  = \legendre{a_1\cdots a_n}{\varpi^k}\legendre{a}{\varpi^k}^n \tau_{\varpi^k}^n.
  \]
As $a$ and $\varpi$ are coprime, we have 
 \begin{align*}
     \legendre{a}{\varpi^k}^n=\legendre{a}{\varpi}^{nk}
     =\begin{dcases}
         1&\textup{if }2\mid n\textup{ or }2\mid k,\\
         \legendre{a}{\varpi}&\textup{else,}
     \end{dcases}
 \end{align*}
and thus
 \begin{align}
  S_{a,\varpi^k}(f)
  =\begin{dcases} 
  \legendre{a_1\cdots a_n}{\varpi^k}\tau_{\varpi^k}^n &\textup{if }2\mid n\textup{ or }2\mid k,\\
  \legendre{a}{\varpi}\legendre{a_1\cdots a_n}{\varpi^k}\tau_{\varpi^k}^n &\textup{else.}
  \end{dcases}\label{gl-S_awk(f)=Fallunterscheidung:nGerade,kGerade,sonst}
 \end{align}
Since the number of all $a\in\fqt$ with $\abs{a}<\abs{\varpi}^k$ and $\gcd(a,\varpi^k)=1$ is given by $\varphi(\varpi^k)$, we obtain
 \begin{align}
    S_{\varpi^k}(f)
    &=\sum_{\substack{a\in\fqt \\ \abs{a}<\abs{\varpi}^k \\ \gcd(a,\varpi^k)=1}} S_{a,\varpi^k}(f)
    =\legendre{a_1\cdots a_n}{\varpi^k} \varphi(\varpi^k)\tau_{\varpi^k}^n\label{gl-S_varpik(f)=...zwischenschritt}
 \end{align}
if $n$ or $k$ is even. For the case where $k$ is even, the Jacobi symbol in~\eqref{gl-S_varpik(f)=...zwischenschritt} is 1 and from Lemma~\ref{lem-tau_varpikPotenz} it then follows
\begin{align*}
    S_{\varpi^k}(f)
    &=\varphi(\varpi^k)\abs{\varpi^k}^{\frac{n}{2}}
    =\legendre{(-1)^{\frac{n}{2}} a_1\cdots a_n}{\varpi^k}\varphi(\varpi^k) \abs{\varpi^k}^{\frac{n}{2}}.
\end{align*}
If $n$ is even and $k$ is odd, we get
\begin{align*}
    S_{\varpi^k}(f)
    &=\legendre{a_1\cdots a_n}{\varpi^k} \varphi(\varpi^k)\left(-i_p^{\,\deg(\varpi)\nu}\abs{\varpi}^{\frac{k}{2}}\right)^n\\
    &=\legendre{a_1\cdots a_n}{\varpi^k}\varphi(\varpi^k)\, i_p^{\,\deg(\varpi)\nu n}\abs{\varpi^k}^{\frac{n}{2}}
\end{align*}
by applying Lemma~\ref{lem-tau_varpikPotenz} to~\eqref{gl-S_varpik(f)=...zwischenschritt}. As
\begin{align*}
    i_p^{\,2}=\begin{dcases}
    (-1)^2=1&\textup{if }p\equiv 1\mod 4,\\
    (-i)^2=-1&\textup{if }p\equiv 3\mod 4,
    \end{dcases}
\end{align*}
the value of $i_p^{\,2}$ coincides with the Legendre symbol $\legendre{-1}{p}$. We then get
\[
  i_p^{\,2\nu}
  =\legendre{-1}{p}^\nu 
  =\legendre{-1}{q}=(-1)^{\frac{q-1}{2}},
\]
and thus, if $n$ is even, it holds
\begin{align*}
    i_p^{\,\deg(\varpi) \nu n}
    =(-1)^{\frac{q-1}{2}\cdot \frac{\deg(\varpi) n}{2}}
    &=\legendre{-1}{\varpi}^{\frac{n}{2}}
    =\legendre{(-1)^{\frac{n}{2}}}{\varpi}
    =\legendre{(-1)^{\frac{n}{2}}}{\varpi^k},
\end{align*}
where we use general properties of the Jacobi symbol for function fields (see~{\cite[Propositions~3.2 and 3.4]{rosen2002}}). Therefore, for the case where $n$ is even and $k$ is odd, we obtain 
\begin{align*}
   S_{\varpi^k}(f)
   &=\legendre{a_1\cdots a_n}{\varpi^k}\varphi(\varpi^k)\legendre{(-1)^{\frac{n}{2}}}{\varpi^k} \abs{\varpi^k}^{\frac{n}{2}}\\
   &=\legendre{(-1)^{\frac{n}{2}} a_1\cdots a_n}{\varpi^k}\varphi(\varpi^k)\abs{\varpi^k}^{\frac{n}{2}}.
\end{align*}
Suppose now that both $n$ and $k$ are odd. Since the Legendre symbol~$\legendre{\cdot}{\varpi}$ is a non-trivial multiplicative character of~$\fqt/\!\left(\varpi^k\right)$, we have
\[
   \sum_{\substack{a\in\fqt \\ \abs{a}<\abs{\varpi}^k \\ \gcd(a,\varpi^k)=1}} \legendre{a}{\varpi}=0.
\]
From~\eqref{gl-S_awk(f)=Fallunterscheidung:nGerade,kGerade,sonst}, it then follows
\begin{align*}
   S_{\varpi^k}(f)
   &=\legendre{a_1\cdots a_n}{\varpi^k}\tau_{\varpi^k}^n\sum_{\substack{a\in\fqt \\ \abs{a}<\abs{\varpi}^k \\ \gcd(a,\varpi^k)=1}} \legendre{a}{\varpi}
   =0.\qedhere
\end{align*}
\end{proof}

Using the factorisation obtained in~\eqref{gl-faktorisierungS_r(f)}, we now get an exact formula for $S_r(f)$ for arbitrary monic $r\in\fqt$.

\begin{thm}\label{thm-S_r(f)}
Let $r\in\fqt$ be monic. If $n$ is even, we have
\[
   S_r(f)=\legendre{(-1)^{\frac{n}{2}}a_1\cdots a_n}{r}\varphi(r)\abs{r}^{\frac{n}{2}},
\]
and if $n$ is odd, we have
\[
    S_r(f)=\begin{dcases}
        \varphi(r)\abs{r}^{\frac{n}{2}} &\textup{if }r\textup{ is a square in }\fqt,\\
        0&\textup{else.}
    \end{dcases}
\]
\end{thm}

\begin{proof}
If $r=1$, we have
\[
    \psi\!\left(\frac{af(\mathbf{b})}{r}\right)=\psi(af(\mathbf{b}))=1
\]
for all $a\in\fqt$ and $\mathbf{b}\in\fqt^n$, and hence
\[
    S_r(f)=\sum_{\substack{a\in\fqt \\ \abs{a}<1\\ \gcd(a,1)=1}}\sum_{\substack{\mathbf{b}\in\fqt^n\\ \abs{\mathbf{b}}<1}} 1=1.
\]
Thus, we see that for~$r=1$ the identity in Theorem~\ref{thm-S_r(f)} holds in both cases since $r$ then is a square in~$\fqt$.

Let now $r\neq 1$. As $r$ is monic, we have $\deg(r)\ge 1$, and $r$ has a prime factorisation
 \[
    r=\varpi^{k_1}_1\cdots \varpi_m^{k_m}
\]
for pairwise distinct monic irreducible elements $\varpi_1,\dots,\varpi_m\in\fqt$ and $k_1,\dots,$ $k_m,$ $m\in\N$. According to~\eqref{gl-faktorisierungS_r(f)}, we can write
\[
    S_r(f)
    =\prod_{i=1}^m S_{\varpi_i^{k_i}}(f).
\]
If $n$ is even, applying Lemma~\ref{lem-S_varpiPotenz} and using that the Jacobi symbol and Euler's totient function are multiplicative yields
\begin{align*}
    S_r(f)
    &=\prod_{i=1}^m \legendre{(-1)^{\frac{n}{2}} a_1\cdots a_n}{\varpi_i^{k_i}}\varphi(\varpi_i^{k_i})\abs{\varpi_i^{k_i}}^{\frac{n}{2}}\\
    &=\left(\prod_{i=1}^m \legendre{(-1)^{\frac{n}{2}}a_1\cdots a_n}{\varpi_i^{k_i}}\right)\left(\prod_{i=1}^m\varphi(\varpi_i^{k_i})\right)\left(\prod_{i=1}^m \abs{\varpi_i^{k_i}}^{\frac{n}{2}}\right)\\
    &=\legendre{(-1)^{\frac{n}{2}}a_1\cdots a_n}{r}\varphi(r)\abs{r}^{\frac{n}{2}}.
\end{align*}

Suppose now that $n$ is odd. As per~\eqref{gl-faktorisierungS_r(f)}, we have $S_r(f)=0$ if and only if there is an $i\in\{1,\dots,m\}$ such that $S_{\varpi_i^{k_i}}(f)=0$, which is according to Lemma~\ref{lem-S_varpiPotenz} equivalent to $k_i$ being odd. Hence, it holds $S_r(f)\neq 0$ if and only if for every $i=1,\dots,m$ the prime factor $\varpi_i$ appears with even exponent $k_i$ in the prime factorisation of $r$, which is equivalent to $r$ being a square in $\fqt$. In this case, we then get 
\[
    S_r(f)
    =\prod_{i=1}^m \varphi\!\left(\varpi_i^{k_i}\right)\abs{\varpi_i^{k_i}}^{\frac{n}{2}}
    =\varphi(r)\abs{r}^{\frac{n}{2}}.
\]
\end{proof}

\section{Exact Formulas for \texorpdfstring{$N(P)$}{N(P)}}\label{kap-berechnungN(P)}
As noticed in~\eqref{gl-N(P)=SumS_r(f)Int}, we have
\begin{align*}
    N(P)
    &=\sum_{\substack{r\in\fqt\\ \abs{r}\le q^P\\ r\textup{ monic}}} \frac{S_r(f)}{\abs{r}^n}\, I_r,
\end{align*}
where for all $r\in\fqt$ the value of the integral $I_r$ is given by
\begin{align*}
    I_r =\begin{dcases}
        \frac{\abs{r}^nq^{n+1}}{q^{2P}}\sum_{k=0}^{P-\deg(r)-1}q^{nk}\,S_{t^{P-\deg(r)-k-1}}(f)&\textup{if }\abs{r}\le q^{P-1},\\
        q^{P(n-2)}&\textup{if }\abs{r}= q^P
    \end{dcases}
\end{align*}
according to Lemma~\ref{lem-Integral=quasiSumGaußsummen}. Since the value of $I_r$ depends on the absolute value of $r$, it will be convenient to write
\begin{align}
    N(P)
    &=\sum_{\substack{r\in\fqt\\ \abs{r}\le q^{P-1}\\ r \textup{ monic}}} \frac{S_r(f)}{\abs{r}^n}\, I_r
    + \sum_{\substack{r\in\fqt\\ \abs{r}=q^{P}\\ r \textup{ monic}}} \frac{S_r(f)}{\abs{r}^n}\, I_r.\label{gl-N(P)=1+2}
\end{align}
We now want to apply the results for $S_r(f)$ that we obtained in Theorem~\ref{thm-S_r(f)}. For this, we distinguish between the three cases
\begin{itemize}
    \item[(1)] $n$ is even and $(-1)^{\frac{n}{2}}a_1\cdots a_n$ is a square in $\fq^\times$,
    \item[(2)] $n$ is even and $(-1)^{\frac{n}{2}}a_1\cdots a_n$ is not a square in $\fq^\times$,
    \item[(3)] $n$ is odd. 
\end{itemize}

\subsection{Computation of \texorpdfstring{$N(P)$}{N(P)} if \texorpdfstring{$2\mid n$}{2|n} and \texorpdfstring{$(-1)^{\frac{n}{2}}a_1\cdots a_n\in (\fq^{\times})^2$}{-1 to the n/2 a1...an is a square in Fq}}
First, we consider the case where $n$ is even and $(-1)^{\frac{n}{2}}a_1\cdots a_n$ is a square in $\fq^\times$, i.e.
\[
   \left((-1)^{\frac{n}{2}}a_1\cdots a_n\right)^{\frac{q-1}{2}}=1.
\]
For all $a\in\fq^\times$ and monic $r\in\fqt$ with prime factorisation $r=\varpi_1^{k_1}\cdots \varpi_m^{k_m}$ for distinct monic irreducible elements $\varpi_1\dots,\varpi_m\in\fqt$ and $m,k_1,\dots,k_m\in\N$, we can evaluate the Jacobi symbol $\legendre{a}{r}$ via
\begin{align}
    \legendre{a}{r}
    &=\prod_{i=1}^m \legendre{a}{\varpi_i}^{k_i}
    =\prod_{i=1}^m a^{\frac{q-1}{2}\,\deg(\varpi_i)k_i}
    =\left(a^{\frac{q-1}{2}}\right)^{\deg(r)}\label{gl-JacobisymbolAuswertung}
\end{align}
by applying general properties of the Jacobi symbol for function fields (see~{\cite[Propositions~3.2 and 3.4]{rosen2002}}). Hence, in our case we have
\[
   \legendre{(-1)^{\frac{n}{2}}a_1\cdots a_n}{r}
   =1.
\]
Theorem~\ref{thm-S_r(f)} therefore yields
\begin{align}
    S_r(f)&=\varphi(r)\abs{r}^{\frac{n}{2}}\label{gl-Quadrat:S_r(f)}
\end{align}
for all monic $r\in\fqt$. In particular, we have
\begin{align}
    S_{t^k}(f)
    &=\varphi(t^k)\abs{t^k}^{\frac{n}{2}}
    =\begin{dcases}
    1&\textup{if }k=0,\\
    \frac{q-1}{q}\, q^{k\left(\frac{n}{2}+1\right)}&\textup{if }k\ge 1
    \end{dcases}\label{gl-Quadrat:S_tc}
\end{align}
for all $k\in\N_0$, where we apply~{\cite[Proposition~1.7]{rosen2002}} to evaluate $\varphi(t^k)$. 

To compute $N(P)$, we will insert~\eqref{gl-Quadrat:S_r(f)} and~\eqref{gl-Quadrat:S_tc} into the formula for $N(P)$ given in~\eqref{gl-N(P)=1+2}. As this includes computing sums of Euler's totient function, we will make use of the following identity~{\cite[Proposition~2.7]{rosen2002}}.

\begin{lem}\label{lem-sum:phi}
Let $\rho\in\N$. We have
\[
   \sum_{\substack{r\in\fqt \\ \deg(r)=\rho\\ r\textup{ monic}}} \varphi(r) = \frac{q-1}{q}\, q^{2\rho}.
\]
\end{lem}

Using Lemma~\ref{lem-sum:phi}, we can directly derive the following lemma.

\begin{lem}\label{lem-geometrischeReihe}
Let $M\in \N_0$ and $c\in\R$. It holds
\begin{align*}
    \sum_{\substack{r\in\fqt\\ \abs{r}\le q^M\\ r\textup{ monic}}} \frac{\varphi(r)}{\abs{r}^c}=
    \begin{dcases}
        1+\frac{q-1}{q}M &\textup{if }c=2,\\
        \frac{1-q^{1-c}}{1-q^{2-c}}-\frac{(q-1)q^{1-c}}{1-q^{2-c}}\, q^{M(2-c)}&\textup{else.}
    \end{dcases}
\end{align*}
Moreover, if $c=2$, we have
\begin{align*}
    \sum_{\substack{r\in\fqt\\ \abs{r}\le q^M\\ r\textup{ monic}}} (-1)^{\deg(r)}\,\frac{\varphi(r)}{\abs{r}^c}=
    \begin{dcases}
        1 &\textup{if }2\mid M,\\
        \frac{1}{q}&\textup{if }2\nmid M,
    \end{dcases}
    \end{align*}
and if $c\neq 2$, we have
\begin{align*}
    \sum_{\substack{r\in\fqt\\ \abs{r}\le q^M\\ r\textup{ monic}}} (-1)^{\deg(r)}\frac{\varphi(r)}{\abs{r}^c}
    &=\frac{1+q^{1-c}}{1+q^{2-c}}+(-1)^M\,\frac{(q-1)q^{1-c}}{1+q^{2-c}}\, q^{M(2-c)}.
\end{align*}
\end{lem}

\begin{proof}
Since $1\in\fqt$ is the only monic element of degree $0$, applying Lemma~\ref{lem-sum:phi} yields 
\begin{align*}
\sum_{\substack{r\in\fqt\\ \abs{r}\le q^M\\ r\textup{ monic}}} \frac{\varphi(r)}{\abs{r}^c}
    &=\sum_{\rho=0}^{M} \frac{1}{q^{c\rho}}\sum_{\substack{r\in\fqt\\ \deg(r)=\rho\\ r\textup{ monic}}} \varphi(r)\\
    &=1+\sum_{\rho=1}^{M} \frac{1}{q^{c\rho}}\cdot\frac{q-1}{q}\, q^{2\rho} \\
    &=1+\frac{q-1}{q}\sum_{\rho=1}^{M}q^{\rho(2-c)}.
\end{align*}
If $c=2$, it immediately follows
\begin{align*}
    \sum_{\substack{r\in\fqt\\ \abs{r}\le q^M\\ r\textup{ monic}}} \frac{\varphi(r)}{\abs{r}^c}
    &=1+\frac{q-1}{q}\,M.
\end{align*}
Now we assume $c\neq 2$. Then, we can evaluate the partial sum of the geometric series via
\begin{align*}
    \sum_{\rho=1}^{M}q^{\rho(2-c)}
    &=q^{2-c}\,\frac{1-q^{M(2-c)}}{1-q^{2-c}},
\end{align*}
and we then obtain
\begin{align*}
    \sum_{\substack{r\in\fqt\\ \abs{r}\le q^M\\ r\textup{ monic}}} \frac{\varphi(r)}{\abs{r}^c}
    &= 1+ \frac{(q-1)q^{1-c}\left(1-q^{M(2-c)}\right)}{1-q^{2-c}}\\
    &=\frac{1-q^{1-c}}{1-q^{2-c}}-\frac{(q-1)q^{1-c}}{1-q^{2-c}}\, q^{M(2-c)}.
\end{align*}
Similarly, we get
\begin{align*}
     \sum_{\substack{r\in\fqt\\ \abs{r}\le q^M\\ r\textup{ monic}}} (-1)^{\deg(r)}\frac{\varphi(r)}{\abs{r}^c}
     &=\sum_{\rho=0}^M\left(-\frac{1}{q^c}\right)^\rho \sum_{\substack{r\in\fqt\\ \deg(r)=\rho\\ r\textup{ monic}}}\varphi (r)\\
     &=1+\frac{q-1}{q}\sum_{\rho=1}^M \left(-q^{2-c}\right)^\rho
\end{align*}
from Lemma~\ref{lem-sum:phi}. If $c=2$, we have
\begin{align*}
    \sum_{\rho=1}^M \left(-q^{2-c}\right)^\rho
    =\sum_{\rho=1}^M (-1)^\rho=\begin{dcases}
        0&\textup{if }2\mid M,\\
        -1&\textup{if }2\nmid M,
    \end{dcases}
\end{align*}
and hence
\begin{align*}
     \sum_{\substack{r\in\fqt\\ \abs{r}\le q^M\\ r\textup{ monic}}} (-1)^{\deg(r)}\,\frac{\varphi(r)}{\abs{r}^c}
     =\begin{dcases}
         1&\textup{if }2\mid M,\\
         1-\frac{q-1}{q}=\frac{1}{q}&\textup{if }2\nmid M.
     \end{dcases}
\end{align*}
Whereas if $c\neq 2$, it follows from 
\begin{align*}
    \sum_{\rho=1}^{M}\left(-q^{2-c}\right)^\rho
    &=-q^{2-c}\,\frac{1-\left(-q^{2-c}\right)^{M}}{1+q^{2-c}}
\end{align*}
that
\begin{align*}
    \sum_{\substack{r\in\fqt\\ \abs{r}\le q^M\\ r\textup{ monic}}} \frac{\varphi(r)}{\abs{r}^c}
    &= 1- \frac{(q-1)q^{1-c}\left(1-(-1)^{M}q^{M(2-c)}\right)}{1+q^{2-c}}\\
    &= \frac{1+q^{1-c}}{1+q^{2-c}}+(-1)^M \frac{(q-1)q^{1-c}}{1+q^{2-c}}\,q^{M(2-c)}.
\end{align*}
\end{proof}

We now return to the setting where $n$ is even and $(-1)^{\frac{n}{2}}a_1\cdots a_n$ is a square in $\fq^\times$. The main step to compute $N(P)$ consists of the following lemma, where we consider the first sum in~\eqref{gl-N(P)=1+2}.

\begin{lem}\label{lem-Quadrat:1.Summe}
Let $n$ be even, and let $(-1)^{\frac{n}{2}}a_1\cdots a_n$ be a square in $\fq^\times$. For $n=4$ we have
\begin{align*}
      \sum_{\substack{r\in\fqt\\ \abs{r}\le q^{P-1}\\ r \textup{ monic}}}\frac{S_r(f)}{\abs{r}^n}\, I_r
    &=\frac{q^2-1}{q}Pq^{2P}+\frac{1}{q}\, q^{2P},
\end{align*}
and for $n\ge 6$ we have
\begin{align*}
    \sum_{\substack{r\in\fqt\\ \abs{r}\le q^{P-1}\\ r \textup{ monic}}}\frac{S_r(f)}{\abs{r}^n}\, I_r
    &=\frac{q^{\frac{n}{2}}-1}{q^{\frac{n}{2}-1}-q}\, q^{P(n-2)}-\frac{q^2-1}{q\left(1-q^{2-\frac{n}{2}}\right)} \, q^{\frac{n}{2}P}.
\end{align*}
\end{lem}

\begin{proof}
    Let $r\in\fqt$ be monic with $\rho\coloneq \deg(r)\le P-1$. Applying~\eqref{gl-Quadrat:S_tc} yields
\begin{align*}
    s_r\coloneq \sum_{k=0}^{P-\rho-1}q^{nk}S_{t^{P-\rho-k-1}}(f)
    &=\sum_{k=0}^{P-\rho-2} q^{nk}\cdot\frac{q-1}{q}\, q^{(P-\rho-k-1)\left(\frac{n}{2}+1\right)}+q^{n(P-\rho-1)}\\
    &=\frac{q-1}{q}\, q^{(P-\rho-1)\left(\frac{n}{2}+1\right)}\sum_{k=0}^{P-\rho-2}\left(q^{\frac{n}{2}-1}\right)^k +q^{n(P-\rho-1)}.
\end{align*}
From
\begin{align*}
    \sum_{k=0}^{P-\rho-2}\left(q^{\frac{n}{2}-1}\right)^k 
    &= \frac{\left(q^{\frac{n}{2}-1}\right)^{P-\rho-1}-1}{q^{\frac{n}{2}-1}-1}
\end{align*}
for $n\ge 4$, we hence get
\begin{align*}
    s_r
    &=\frac{q-1}{q}\cdot\frac{q^{-\left(\frac{n}{2}+1\right)}q^{1-\frac{n}{2}}}{q^{\frac{n}{2}-1}-1}\, q^{nP-n\rho} -\frac{q-1}{q}\cdot\frac{q^{-\left(\frac{n}{2}+1\right)}}{q^{\frac{n}{2}-1}-1}\, q^{P\left(\frac{n}{2}-1\right)-\rho\left(\frac{n}{2}-1\right)}+q^{n(P-\rho-1)}\\
    &=\left(\frac{q-1}{q^{n+1}\left(q^{\frac{n}{2}-1}-1\right)}+\frac{1}{q^n}\right) q^{nP-n\rho} -\frac{q-1}{q^{\frac{n}{2}+2}\left(q^{\frac{n}{2}-1}-1\right)}\,q^{P\left(\frac{n}{2}+1\right)-\rho\left(\frac{n}{2}+1\right)}.
\end{align*}
It then follows
\begin{align*}
    I_r=\frac{q^{n\rho+n+1}}{q^{2P}}\cdot s_r
    &=\left(\frac{q-1}{q^{\frac{n}{2}-1}-1}+q\right) q^{P(n-2)} -\frac{(q-1)q^{\frac{n}{2}-1}}{q^{\frac{n}{2}-1}-1}\, q^{P\left(\frac{n}{2}-1\right)+\rho\left(\frac{n}{2}-1\right)} \\
    &=\frac{q^{\frac{n}{2}}-1}{q^{\frac{n}{2}-1}-1}\, q^{P(n-2)} -\frac{(q-1)q^{\frac{n}{2}-1}}{q^{\frac{n}{2}-1}-1}\, q^{P\left(\frac{n}{2}-1\right)+\rho\left(\frac{n}{2}-1\right)}.
\end{align*}
Therefore, from~\eqref{gl-Quadrat:S_r(f)} we obtain
\begin{align}
    \sum_{\substack{r\in\fqt\\ \abs{r}\le q^{P-1}\\ r\textup{ monic}}}\frac{S_r(f)}{\abs{r}^n} \, I_r
    &=\sum_{\substack{r\in\fqt\\ \abs{r}\le q^{P-1}\\ r\textup{ monic}}} \frac{\varphi(r)}{\abs{r}^{\frac{n}{2}}}\cdot \frac{q^{\frac{n}{2}}-1}{q^{\frac{n}{2}-1}-1}\, q^{P(n-2)} \notag\\
    &\quad - \sum_{\substack{r\in\fqt\\ \abs{r}\le q^{P-1}\\ r\textup{ monic}}} \frac{\varphi(r)}{\abs{r}^{\frac{n}{2}}}\cdot \frac{(q-1)q^{\frac{n}{2}-1}}{q^{\frac{n}{2}-1}-1}\,q^{P\left(\frac{n}{2}-1\right)}\abs{r}^{\frac{n}{2}-1}\notag\\
    &= \frac{q^{\frac{n}{2}}-1}{q^{\frac{n}{2}-1}-1}\, q^{P(n-2)}\sum_{\substack{r\in\fqt\\ \abs{r}\le q^{P-1}\\ r\textup{ monic}}} \frac{\varphi(r)}{\abs{r}^{\frac{n}{2}}}  -\frac{(q-1)q^{\frac{n}{2}-1}}{q^{\frac{n}{2}-1}-1}\, q^{P\left(\frac{n}{2}-1\right)} \sum_{\substack{r\in\fqt\\ \abs{r}\le q^{P-1}\\ r\textup{ monic}}}\frac{\varphi(r)}{\abs{r}}.\label{gl-Quadrat:SumS_r(f)I_rGrad<=P-1}
\end{align}
According to Lemma~\ref{lem-geometrischeReihe}, we have
\begin{align}
    \sum_{\substack{r\in\fqt\\ \abs{r}\le q^{P-1}\\ r\textup{ monic}}}\frac{\varphi(r)}{\abs{r}}
    &=-\frac{q-1}{1-q}\cdot q^{P-1}
    =q^{P-1}\label{gl-geomsum:phi=qP-1}
\end{align}
and
\begin{align*}
    \sum_{\substack{r\in\fqt\\ \abs{r}\le q^{P-1}\\ r\textup{ monic}}} \frac{\varphi(r)}{\abs{r}^{\frac{n}{2}}}
    &=\begin{dcases}
        1+\frac{q-1}{q}(P-1)&\textup{if }n=4,\\
        \frac{1-q^{1-\frac{n}{2}}}{1-q^{2-\frac{n}{2}}}-\frac{(q-1)q^{1-\frac{n}{2}}}{1-q^{2-\frac{n}{2}}}\,q^{P\left(2-\frac{n}{2}\right)+\frac{n}{2}-2} &\textup{if }n\ge 6
    \end{dcases}\\
    &=\begin{dcases}
        \frac{1}{q}+\frac{q-1}{q}\,P &\textup{if }n=4,\\
        \frac{1-q^{1-\frac{n}{2}}}{1-q^{2-\frac{n}{2}}}-\frac{q-1}{q\left(1-q^{2-\frac{n}{2}}\right)}\,q^{P\left(2-\frac{n}{2}\right)} &\textup{if }n\ge 6.
    \end{dcases}
\end{align*}
Hence, if $n=4$, inserting into~\eqref{gl-Quadrat:SumS_r(f)I_rGrad<=P-1} yields
\begin{align*}
 \sum_{\substack{r\in\fqt\\ \abs{r}\le q^{P-1}\\ r\textup{ monic}}}\frac{S_r(f)}{\abs{r}^n} \, I_r
 &=\frac{q^{2}-1}{q-1}\, q^{2P}\left(\frac{1}{q}+\frac{q-1}{q}P\right)-\frac{(q-1)q}{q-1}\, q^{2P-1}\\
    &= \frac{q^{2}-1}{q}Pq^{2P}+\left(\frac{q+1}{q}-1\right)q^{2P}\\
    &= \frac{q^2-1}{q}Pq^{2P}+\frac{1}{q}\,q^{2P}.
\end{align*}
Whereas if $n\ge 6$, we obtain
\begin{align*}
     \sum_{\substack{r\in\fqt\\ \abs{r}\le q^{P-1}\\ r\textup{ monic}}}\frac{S_r(f)}{\abs{r}^n} \, I_r
    =&\: \frac{q^{\frac{n}{2}}-1}{q^{\frac{n}{2}-1}-1}\, q^{P(n-2)}\left(\frac{1-q^{1-\frac{n}{2}}}{1-q^{2-\frac{n}{2}}}-\frac{q-1}{q\left(1-q^{2-\frac{n}{2}}\right)}\,q^{P\left(2-\frac{n}{2}\right)}\right)\\
    &\:-\frac{(q-1)q^{\frac{n}{2}-1}}{q^{\frac{n}{2}-1}-1}\, q^{P\left(\frac{n}{2}-1\right)}\, q^{P-1}\\
    =&\: \frac{q^{\frac{n}{2}}-1}{q^{\frac{n}{2}-1}\left(1-q^{2-\frac{n}{2}}\right)}\, q^{P(n-2)}-\left(\frac{(q-1)\left(q^{\frac{n}{2}}-1\right)}{q\left(q^{\frac{n}{2}-1}-1\right)\left(1-q^{2-\frac{n}{2}}\right)}+\frac{(q-1)q^{\frac{n}{2}-2}}{q^{\frac{n}{2}-1}-1}\right) q^{\frac{n}{2}P}
\end{align*}
from~\eqref{gl-Quadrat:SumS_r(f)I_rGrad<=P-1}, where the coefficient of $q^{\frac{n}{2}P}$ is then given by the negative of
\begin{align*}
    \frac{q-1}{q\left(q^{\frac{n}{2}-1}-1\right)}\left( \frac{q^{\frac{n}{2}}-1}{1-q^{2-\frac{n}{2}}}+q^{\frac{n}{2}-1}\right)
    &=\frac{(q-1)\left(q^{\frac{n}{2}}+q^{\frac{n}{2}-1}-q-1\right)}{q\left(q^{\frac{n}{2}-1}-1\right)\left(1-q^{2-\frac{n}{2}}\right)}
    =\frac{(q-1)(q+1)}{q\left(1-q^{2-\frac{n}{2}}\right)}.
\end{align*}
Overall, we thus get
\begin{align*}
    \sum_{\substack{r\in\fqt\\ \abs{r}\le q^{P-1}\\ r\textup{ monic}}}\frac{S_r(f)}{\abs{r}^n} \, I_r
    &=\frac{q^{\frac{n}{2}}-1}{q^{\frac{n}{2}-1}-q}\, q^{P(n-2)}-\frac{q^2-1}{q\left(1-q^{2-\frac{n}{2}}\right)} \, q^{\frac{n}{2}P}
\end{align*}
for $n\ge 6$.
\end{proof}

From this, we obtain a formula for $N(P)$ for $n\ge 4$.

\begin{thm}\label{thm-Quadrat:N(P)}
Let $n$ be even and let $(-1)^{\frac{n}{2}}a_1\cdots a_n$ be a square in $\fq^\times$. For $n=4$ we have
\begin{align*}
    N(P)&=\frac{q^2-1}{q}\, Pq^{2P}+q^{2P},
\end{align*}
and for $n\ge 6$ we have
\begin{align*}
    N(P)&=\frac{q^{\frac{n}{2}}-1}{q^{\frac{n}{2}-1}-q} \, q^{P(n-2)}-(q-1)\frac{q^{\frac{n}{2}-1}+1}{q^{\frac{n}{2}-1}-q}\, q^{\frac{n}{2}P}.
\end{align*}
\end{thm}

\begin{proof}
In Lemma~\ref{lem-Quadrat:1.Summe} we have already computed the first sum in~\eqref{gl-N(P)=1+2}, so we now need to look at the second sum. From Lemma~\ref{lem-Integral=quasiSumGaußsummen}, we know that
\[
    I_r=q^{P(n-2)}
\]
for all monic $r\in\fqt$ with $\deg(r)=P$, and hence, from~\eqref{gl-Quadrat:S_r(f)} it follows
\begin{align*}
    \sum_{\substack{r\in\fqt\\ \abs{r}=q^P\\ r\textup{ monic}}} \frac{S_r(f)}{\abs{r}^n}\, I_r
    &= \sum_{\substack{r\in\fqt\\ \abs{r}=q^P\\ r\textup{ monic}}} \frac{\varphi(r)\left(q^P\right)^{\frac{n}{2}}}{\left(q^P\right)^n}\,q^{P(n-2)}
    =q^{P\left(\frac{n}{2}-2\right)}\sum_{\substack{r\in\fqt\\ \abs{r}=q^P\\ r\textup{ monic}}}\varphi(r).
\end{align*}
Applying Lemma~\ref{lem-sum:phi} thus yields
\begin{align*}
    \sum_{\substack{r\in\fqt\\ \abs{r}=q^P\\ r\textup{ monic}}} \frac{S_r(f)}{\abs{r}^n}\, I_r
    &=q^{P\left(\frac{n}{2}-2\right)}\, \frac{q-1}{q}\,q^{2P}
    =\frac{q-1}{q}\, q^{\frac{n}{2}P},
\end{align*}
which we can insert into~\eqref{gl-N(P)=1+2}. For $n=4$ we then obtain
\begin{align*}
    N(P)
    &=\frac{q^2-1}{q}Pq^{2P}+\frac{1}{q}\, q^{2P} + \frac{q-1}{q}\, q^{2P}
    =\frac{q^2-1}{q}Pq^{2P}+q^{2P},
\end{align*}
and for $n\ge 6$ we get
\begin{align*}
    N(P)
    &=\frac{q^{\frac{n}{2}}-1}{q^{\frac{n}{2}-1}-q}\, q^{P(n-2)}-\frac{q^2-1}{q\left(1-q^{2-\frac{n}{2}}\right)}  \, q^{\frac{n}{2}P}
    + \frac{q-1}{q}\,q^{\frac{n}{2}P}.
\end{align*}
Computing
\begin{align}
\frac{q^2-1}{q\left(1-q^{2-\frac{n}{2}}\right)}-\frac{q-1}{q}
    &=\frac{q-1}{q}\left(\frac{q+1}{1-q^{2-\frac{n}{2}}}-1\right)\notag\\
    &=\frac{q-1}{q}\cdot \frac{q+q^{2-\frac{n}{2}}}{1-q^{2-\frac{n}{2}}}
    =(q-1)\frac{q^{\frac{n}{2}-1}+1}{q^{\frac{n}{2}-1}-q}\label{gl-brücheSumme}
\end{align}
then gives
\begin{align*}
    N(P)
    &=\frac{q^{\frac{n}{2}}-1}{q^{\frac{n}{2}-1}-q}\, q^{P(n-2)}- (q-1)\frac{q^{\frac{n}{2}-1}+1}{q^{\frac{n}{2}-1}-q}\,q^{\frac{n}{2}P}.
\end{align*}
\end{proof}

We can now prove the main result, which was stated in Theorem~\ref{thm-quadrat:mor}.

\begin{proof}[Proof of Theorem~\ref{thm-quadrat:mor}]
To compute $\#\mor_P(\PP^1,X)(\fq)$, we insert the formulas from Theorem~\ref{thm-Quadrat:N(P)} into~\eqref{gl-mor_P}. For~$n=4$ this yields
\begin{align*}
    \#\mor_P(\PP^1,X)(\fq)
    &=\frac{q+1}{q}(P+1)q^{2(P+1)}+\frac{1}{q-1}\,q^{2(P+1)}-\frac{(q+1)^2}{q}Pq^{2P}-\frac{q+1}{q-1}\,q^{2P}\\
    &\quad +(q+1)(P-1)q^{2(P-1)}+\frac{q}{q-1}\,q^{2(P-1)}\\
    &=\frac{q^3(q+1)-q(q+1)^2+q+1}{q^2} Pq^{2P}\\
    &\quad+\frac{q^3(q^2-1)+q^4-q^2(q+1)-(q^2-1)+q}{q^2(q-1)}\,q^{2P}\\
    &=\frac{\left(q^2-1\right)^2}{q^2}Pq^{2P}+\frac{\left(q^2-1\right)\left(q+1\right)^2}{q^2}\,q^{2P}.
\end{align*}
If $n\ge 6$, we have
\begin{align*}
    \frac{N(P+1)}{q-1}
    &=\frac{\left(q^{\frac{n}{2}}-1\right)q^{n-2}}{\left(q^{\frac{n}{2}-1}-q\right)(q-1)}\, q^{P(n-2)}-\frac{\left(q^{\frac{n}{2}-1}+1\right)q^{\frac{n}{2}}}{q^{\frac{n}{2}-1}-q}\,q^{\frac{n}{2}P},\\
    -\frac{(q+1)N(P)}{q-1}
    &=-\frac{\left(q^{\frac{n}{2}}-1\right)(q+1)}{\left(q^{\frac{n}{2}-1}-q\right)(q-1)}\, q^{P(n-2)}+\frac{\left(q^{\frac{n}{2}-1}+1\right)(q+1)}{q^{\frac{n}{2}-1}-q}\,q^{\frac{n}{2}P}\\
\intertext{as well as}
    \frac{qN(P-1)}{q-1}
    &=\frac{\left(q^{\frac{n}{2}}-1\right)q^{3-n}}{\left(q^{\frac{n}{2}-1}-q\right)(q-1)}\, q^{P(n-2)}-\frac{\left(q^{\frac{n}{2}-1}+1\right)q^{1-\frac{n}{2}}}{q^{\frac{n}{2}-1}-q}\,q^{\frac{n}{2}P}.
\end{align*}
By summing, we hence obtain
\begin{align*}
    \#\mor_P(\PP^1,X)(\fq)
    &=\frac{\left(q^{\frac{n}{2}}-1\right)\left(q^{n-2}-q-1+q^{3-n}\right)}{\left(q^{\frac{n}{2}-1}-q\right)(q-1)}\, q^{P(n-2)}\\
    &\quad-\frac{\left(q^{\frac{n}{2}-1}+1\right)\left(q^{\frac{n}{2}}-q-1+q^{1-\frac{n}{2}}\right)}{q^{\frac{n}{2}-1}-q}\,q^{\frac{n}{2}P}\\
    &=\frac{\left(q^{\frac{n}{2}}-1\right)\left(q^{n-2}-1\right)\left(1-q^{3-n}\right)}{\left(q^{\frac{n}{2}-1}-q\right)(q-1)}\, q^{P(n-2)}
    - \frac{\left(q^{n-1}-q\right)\left(1-q^{-\frac{n}{2}}\right)}{q^{\frac{n}{2}-1}-q}\,q^{\frac{n}{2}P}\\
    &=\frac{\left(q^{\frac{n}{2}}-1\right)\left(q^{n-2}-1\right)\left(q^{n-3}-1\right)}{q^{n-2}\left(q^{\frac{n}{2}-2}-1\right)(q-1)}\, q^{P(n-2)}
    - \frac{\left(q^{n-2}-1\right)\left(q^{\frac{n}{2}}-1\right)}{q^{\frac{n}{2}}\left(q^{\frac{n}{2}-2}-1\right)}\,q^{\frac{n}{2}P}.
\end{align*}
\end{proof}

\subsection{Computation of \texorpdfstring{$N(P)$}{N(P)} if \texorpdfstring{$2\mid n$}{2|n} and \texorpdfstring{$(-1)^{\frac{n}{2}}a_1\cdots a_n\notin (\fq^{\times})^2$}{-1 to the n/2 a1...an is not a square in Fq}}
Now we consider the case where $n$ is even and $(-1)^{\frac{n}{2}}a_1\cdots a_n$ is not a square in $\fq^\times$, i.e.
\[
    \left((-1)^{\frac{n}{2}}a_1\cdots a_n\right)^{\frac{q-1}{2}}=-1.
\]
Since for all monic $r\in\fqt$ we then have
\[
   \legendre{(-1)^{\frac{n}{2}}a_1\cdots a_n}{r}
   =(-1)^{\deg(r)}
\]
as per~\eqref{gl-JacobisymbolAuswertung}, we obtain
\begin{align}
    S_r(f)
    &=(-1)^{\deg(r)}\varphi(r)\abs{r}^{\frac{n}{2}}\label{gl-keinQuadrat:S_r(f)}
\end{align}
from Theorem~\ref{thm-S_r(f)}. In particular, we have
\begin{align}
    S_{t^k}(f)
    &=(-1)^k\varphi(t^k)\abs{t^k}^{\frac{n}{2}}
    =\begin{dcases}
        1&\textup{if }k=0,\\
    \frac{q-1}{q}\left(-q^{\frac{n}{2}+1}\right)^k&\textup{if }k\ge 1
    \end{dcases}\label{gl-keinQuadrat:S_tc}
\end{align}
for all $k\in\N_0$.

As we did in the previous case in Lemma~\ref{lem-Quadrat:1.Summe}, we begin by computing the first sum in~\eqref{gl-N(P)=1+2}.

\begin{lem}\label{lem-keinQuadrat:1.Summe}
Let $n$ be even, and let $(-1)^{\frac{n}{2}}a_1\cdots a_n$ not be a square in $\fq^\times$. For $n=4$ we have
\begin{align*}
    \sum_{\substack{r\in\fqt\\ \abs{r}\le q^{P-1}\\ r \textup{ monic}}}\frac{S_r(f)}{\abs{r}^n}\, I_r
    =\begin{dcases}
        q^{2P-1}&\textup{if }2\mid P,\\
        q^{2P+1}&\textup{if }2\nmid P,
    \end{dcases}
\end{align*}
and for $n\ge 6$ we have
\begin{align*}
    \sum_{\substack{r\in\fqt\\ \abs{r}\le q^{P-1}\\ r \textup{ monic}}}\frac{S_r(f)}{\abs{r}^n}\, I_r
    &=\frac{q^{\frac{n}{2}}+1}{q^{\frac{n}{2}-1}+q}\, q^{P(n-2)}-(-1)^P\frac{q^2-1}{q\left(1+q^{2-\frac{n}{2}}\right)} \, q^{\frac{n}{2}P}.
\end{align*}
\end{lem}

\begin{proof}
Let $r\in\fqt$ be monic with $\rho\coloneq \deg(r)\le P-1$. From~\eqref{gl-keinQuadrat:S_tc}, we get
\begin{align*}
    s_r\coloneq \sum_{k=0}^{P-\rho-1}q^{nk}S_{t^{P-\rho-k-1}}(f)
    &=\sum_{k=0}^{P-\rho-2} q^{nk}\cdot\frac{q-1}{q}\, \left(-q^{\frac{n}{2}+1}\right)^{P-\rho-k-1}+q^{n(P-\rho-1)}\\
    &=\frac{q-1}{q}\left(-q^{\frac{n}{2}+1}\right)^{P-\rho-1}\sum_{k=0}^{P-\rho-2}\left(-q^{\frac{n}{2}-1}\right)^k+q^{n(P-\rho-1)}.
\end{align*}
Since
\begin{align*}
    \sum_{k=0}^{P-\rho-2}\left(-q^{\frac{n}{2}-1}\right)^k 
    &= -\frac{\left(-q^{\frac{n}{2}-1}\right)^{P-\rho-1}-1}{q^{\frac{n}{2}-1}+1}
\end{align*}
for $n\ge 4$, it then follows
\begin{align*}
    s_r
    &=-\frac{q-1}{q}\cdot \frac{\left(-q^{\frac{n}{2}+1}\left(-q^{\frac{n}{2}-1}\right)\right)^{P-\rho-1}}{q^{\frac{n}{2}-1}+1}+\frac{q-1}{q}\cdot\frac{\left(-q^{\frac{n}{2}+1}\right)^{P-\rho-1}}{q^{\frac{n}{2}-1}+1} +q^{n(P-\rho-1)}\\
    &=\left(\frac{1}{q^n}-\frac{q-1}{q^{n+1}\left(q^{\frac{n}{2}-1}+1\right)}\right) q^{nP-n\rho} +\frac{(-1)^{P-\rho-1}(q-1)}{q^{\frac{n}{2}+2}\left(q^{\frac{n}{2}-1}+1\right)}\,q^{P\left(\frac{n}{2}+1\right)-\rho\left(\frac{n}{2}+1\right)}.
\end{align*}
As in the proof of Lemma~\ref{lem-Quadrat:1.Summe}, we can now multiply $s_r$ with $\frac{q^{n\rho+n+1}}{q^{2P}}$ such that we obtain
\begin{align*}
    I_r
    &=\left(q-\frac{q-1}{q^{\frac{n}{2}-1}+1}\right) q^{P(n-2)} +(-1)^{P-\rho-1}\,\frac{(q-1)q^{\frac{n}{2}-1}}{q^{\frac{n}{2}-1}+1}\, q^{P\left(\frac{n}{2}-1\right)+\rho\left(\frac{n}{2}-1\right)}\\
    &=\frac{q^{\frac{n}{2}}+1}{q^{\frac{n}{2}-1}+1}\, q^{P(n-2)} +(-1)^{P-\rho-1}\,\frac{(q-1)q^{\frac{n}{2}-1}}{q^{\frac{n}{2}-1}+1}\, q^{P\left(\frac{n}{2}-1\right)+\rho\left(\frac{n}{2}-1\right)}.
\end{align*}
Applying~\eqref{gl-keinQuadrat:S_r(f)} then yields
\begin{align}
    \sum_{\substack{r\in\fqt\\ \abs{r}\le q^{P-1}\\ r\textup{ monic}}}\frac{S_r(f)}{\abs{r}^n} \, I_r
    &=\sum_{\substack{r\in\fqt\\ \abs{r}\le q^{P-1}\\ r\textup{ monic}}} (-1)^{\deg(r)}\,\frac{\varphi(r)}{\abs{r}^{\frac{n}{2}}}\cdot \frac{q^{\frac{n}{2}}+1}{q^{\frac{n}{2}-1}+1}\, q^{P(n-2)} \notag\\
    &\quad +\sum_{\substack{r\in\fqt\\ \abs{r}\le q^{P-1}\\ r\textup{ monic}}} (-1)^{\deg(r)}\,\frac{\varphi(r)}{\abs{r}^{\frac{n}{2}}} \cdot  \frac{(-1)^{P-\deg(r)-1}(q-1)q^{\frac{n}{2}-1}}{q^{\frac{n}{2}-1}+1}\,q^{P\left(\frac{n}{2}-1\right)}\abs{r}^{\frac{n}{2}-1}\notag\\
    &= \frac{q^{\frac{n}{2}}+1}{q^{\frac{n}{2}-1}+1}\, q^{P(n-2)}\sum_{\substack{r\in\fqt\\ \abs{r}\le q^{P-1}\\ r\textup{ monic}}} (-1)^{\deg(r)}\,\frac{\varphi(r)}{\abs{r}^{\frac{n}{2}}}\notag \\
    &\quad +(-1)^{P-1}\,\frac{(q-1)q^{\frac{n}{2}-1}}{q^{\frac{n}{2}-1}+1}\, q^{P\left(\frac{n}{2}-1\right)} \sum_{\substack{r\in\fqt\\ \abs{r}\le q^{P-1}\\ r\textup{ monic}}}\frac{\varphi(r)}{\abs{r}}.\label{gl-keinQuadratSumS_r(f)I_r:Grad<=P-1}
\end{align}
According to Lemma~\ref{lem-geometrischeReihe}, we have
\begin{align*}
    \sum_{\substack{r\in\fqt\\ \abs{r}\le q^{P-1}\\ r\textup{ monic}}} (-1)^{\deg(r)}\,\frac{\varphi(r)}{\abs{r}^{\frac{n}{2}}}
    &=\begin{dcases}
        \frac{1}{q}&\textup{if }n=4\textup{ and }2\mid P,\\
        1&\textup{if }n=4\textup{ and }2\nmid P,
    \end{dcases}\end{align*}
    as well as
\begin{align*}
    \sum_{\substack{r\in\fqt\\ \abs{r}\le q^{P-1}\\ r\textup{ monic}}} (-1)^{\deg(r)}\,\frac{\varphi(r)}{\abs{r}^{\frac{n}{2}}}
    &=\frac{1+q^{1-\frac{n}{2}}}{1+q^{2-\frac{n}{2}}}+(-1)^{P-1}\,\frac{(q-1)q^{1-\frac{n}{2}}}{1+q^{2-\frac{n}{2}}}\,q^{(P-1)\left(2-\frac{n}{2}\right)}\\
    &=\frac{1+q^{1-\frac{n}{2}}}{1+q^{2-\frac{n}{2}}}-(-1)^P\,\frac{q-1}{q\left(1+q^{2-\frac{n}{2}}\right)}\,q^{P\left(2-\frac{n}{2}\right)} 
\end{align*}
for $n\ge 6$. Together with~\eqref{gl-geomsum:phi=qP-1}, inserting into~\eqref{gl-keinQuadratSumS_r(f)I_r:Grad<=P-1} gives
\begin{align*}
   \sum_{\substack{r\in\fqt\\ \abs{r}\le q^{P-1}\\ r\textup{ monic}}}\frac{S_r(f)}{\abs{r}^n} \, I_r
    &=\begin{dcases} \frac{q^{2}+1}{q+1}\cdot \frac{1}{q}\, q^{2P}-\frac{(q-1)q}{q+1}\, q^{2P-1}&\textup{if }2\mid P,\\
 \frac{q^{2}+1}{q+1}\, q^{2P}+\frac{(q-1)q}{q+1}\, q^{2P-1}&\textup{if }2\nmid P,
    \end{dcases}\\
    &=\begin{dcases} 
        \frac{1+q}{q(q+1)}\,q^{2P}=q^{2P-1}&\textup{if }2\mid P,\\
        \frac{q^2+q}{q+1}\,q^{2P}=q^{2P+1}&\textup{if }2\nmid P
    \end{dcases}
\end{align*}
for $n=4$, whereas if $n\ge 6$, we obtain
\begin{align*}
     \sum_{\substack{r\in\fqt\\ \abs{r}\le q^{P-1}\\ r\textup{ monic}}}\frac{S_r(f)}{\abs{r}^n} \, I_r
    =&\: \frac{q^{\frac{n}{2}}+1}{q^{\frac{n}{2}-1}+1}\, q^{P(n-2)}\left(\frac{1+q^{1-\frac{n}{2}}}{1+q^{2-\frac{n}{2}}}-(-1)^P\,\frac{q-1}{q\left(1+q^{2-\frac{n}{2}}\right)}q^{P\left(2-\frac{n}{2}\right)}\right)\\
    &+(-1)^{P-1}\,\frac{(q-1)q^{\frac{n}{2}-1}}{q^{\frac{n}{2}-1}+1}\, q^{P\left(\frac{n}{2}-1\right)}\,q^{P-1}\\
    =&\: \frac{q^{\frac{n}{2}}+1}{q^{\frac{n}{2}-1}\left(1+q^{2-\frac{n}{2}}\right)}\, q^{P(n-2)}-(-1)^{P}\frac{(q-1)}{q\left(q^{\frac{n}{2}-1}+1\right)}\left(\frac{q^{\frac{n}{2}}+1}{1+q^{2-\frac{n}{2}}}+q^{\frac{n}{2}-1}\right) q^{\frac{n}{2}P}.
\end{align*}
From
\begin{align*}
    \frac{(q-1)}{q\left(q^{\frac{n}{2}-1}+1\right)}\left(\frac{q^{\frac{n}{2}}+1}{1+q^{2-\frac{n}{2}}}+q^{\frac{n}{2}-1}\right)
    &=\frac{(q-1)\left(q^{\frac{n}{2}}+q^{\frac{n}{2}-1}+q+1\right)}{q\left(q^{\frac{n}{2}-1}+1\right)\left(1+q^{2-\frac{n}{2}}\right)}
    =\frac{(q-1)(q+1)}{q\left(1+q^{2-\frac{n}{2}}\right)},
\end{align*}
overall it then follows 
\begin{align*}
    \sum_{\substack{r\in\fqt\\ \abs{r}\le q^{P-1}\\ r\textup{ monic}}}\frac{S_r(f)}{\abs{r}^n} \, I_r
    &=\frac{q^{\frac{n}{2}}+1}{q^{\frac{n}{2}-1}+q}\, q^{P(n-2)}-(-1)^P\frac{q^2-1}{q\left(1+q^{2-\frac{n}{2}}\right)} \, q^{\frac{n}{2}P}
\end{align*}
for $n\ge 6$.
\end{proof}

We can now compute $N(P)$, as in Theorem~\ref{thm-Quadrat:N(P)} in the previous case.

\begin{thm}\label{thm-N(P)keinQuadrat}
Let $n$ be even, and let $(-1)^{\frac{n}{2}}a_1\cdots a_n$ not be a square in $\fq^\times$. For $n=4$ we have
\begin{align*}
   N(P)
   =\begin{dcases}
   q^{2P} &\textup{if }2\mid P,\\
   \frac{q^2-q+1}{q}\, q^{2P} &\textup{if }2\nmid P,
   \end{dcases}
\end{align*}
and for $n\ge 6$ we have
\begin{align*}
   N(P)
   &=\frac{q^{\frac{n}{2}}+1}{q^{\frac{n}{2}-1}+q} \, q^{P(n-2)}-(-1)^P(q-1)\frac{q^{\frac{n}{2}-1}-1}{q^{\frac{n}{2}-1}+q}\, q^{\frac{n}{2}P}.
\end{align*}
\end{thm}

\begin{proof}
As in the proof of Theorem~\ref{thm-Quadrat:N(P)}, we compute the second sum in~\eqref{gl-N(P)=1+2}. From Lemma~\ref{lem-Integral=quasiSumGaußsummen} and~\eqref{gl-keinQuadrat:S_r(f)}, we get
\begin{align*}
    \sum_{\substack{r\in\fqt\\ \abs{r}=q^P\\ r\textup{ monic}}} \frac{S_r(f)}{\abs{r}^n}\,I_r
    &= q^{P(n-2)}\sum_{\substack{r\in\fqt\\ \abs{r}=q^P\\ r\textup{ monic}}} (-1)^P\,\frac{\varphi(r)\left(q^P\right)^{\frac{n}{2}}}{\left(q^P\right)^n}\\
    &=(-1)^Pq^{P\left(\frac{n}{2}-2\right)}\sum_{\substack{r\in\fqt\\ \abs{r}=q^P\\ r\textup{ monic}}}\varphi(r),
\end{align*}
and hence, applying Lemma~\ref{lem-sum:phi} gives
\begin{align*}
    \sum_{\substack{r\in\fqt\\ \abs{r}=q^\rho\\ r\textup{ monic}}} \frac{S_r(f)}{\abs{r}^n}\,I_r
    &= (-1)^P\,\frac{q-1}{q}\, q^{\frac{n}{2}P}.
\end{align*}
By inserting into~\eqref{gl-N(P)=1+2} and applying Lemma~\ref{lem-keinQuadrat:1.Summe}, we then obtain 
\begin{align*}
    N(P)&=
    \begin{dcases}
       q^{2P-1} + \frac{q-1}{q}\, q^{2P}=q^{2P}&\textup{if }2\mid P,\\
       q^{2P+1}-\frac{q-1}{q}\,q^{2P}=\frac{q^2-q+1}{q}\,q^{2P}&\textup{if }2\nmid P
    \end{dcases}
\end{align*}
for $n=4$, and
\begin{align*}
    N(P)&=\frac{q^{\frac{n}{2}}+1}{q^{\frac{n}{2}-1}+q}\, q^{P(n-2)}-(-1)^P\frac{q^2-1}{q\left(1+q^{2-\frac{n}{2}}\right)} \, q^{\frac{n}{2}P}+(-1)^P\,\frac{q-1}{q}\,q^{\frac{n}{2}P}
\end{align*}
for $n\ge 6$. Analogously to~\ref{gl-brücheSumme}, we can compute
\begin{align*}
    \frac{q^2-1}{q\left(1+q^{2-\frac{n}{2}}\right)} - \frac{q-1}{q}
    &=(q-1)\frac{q^{\frac{n}{2}-1}-1}{q^{\frac{n}{2}-1}+q},
\end{align*}
and hence, we have
\begin{align*}
    N(P)
    &=\frac{q^{\frac{n}{2}}+1}{q^{\frac{n}{2}-1}+q}\, q^{P(n-2)}-(-1)^P (q-1)\frac{q^{\frac{n}{2}-1}-1}{q^{\frac{n}{2}-1}+q}\,q^{\frac{n}{2}P}
\end{align*}
for $n\ge 6$.
\end{proof}

Using Theorem~\ref{thm-N(P)keinQuadrat}, we can now prove Theorem~\ref{thm-keinQUADRAT:mor}.

\begin{proof}[Proof of Theorem~\ref{thm-keinQUADRAT:mor}]
As in the proof of Theorem~\ref{thm-quadrat:mor}, we insert the results from Theorem~\ref{thm-N(P)keinQuadrat} into~\eqref{gl-mor_P}. We first assume $n=4$. If $P$ is even, $P-1$ and $P+1$ are odd, and we get
\begin{align*}
    \#\mor_P(\PP^1,X)(\fq)
    &=\frac{q^2-q+1}{q(q-1)}\,q^{2(P+1)}-\frac{q+1}{q-1}\,q^{2P}+\frac{q^2-q+1}{q-1}\,q^{2(P-1)}\\
    &=\frac{q^3\left(q^2-q+1\right)-q^2(q+1)+q^2-q+1}{q^2(q-1)}\,q^{2P}\\
    &=\frac{q^5-q^4-q+1}{q^2(q-1)}\,q^{2P}=\frac{q^4-1}{q^2}\,q^{2P}.
\end{align*}
Whereas if $P$ is odd, we get
\begin{align*}
    \#\mor_P(\PP^1,X)(\fq)
    &=\frac{1}{q-1}\,q^{2(P+1)}-\frac{(q+1)\left(q^2-q+1\right)}{q(q-1)}\,q^{2P}+\frac{q}{q-1}\, q^{2(P-1)}\\
    &=\frac{q^3-(q+1)\left(q^2-q+1\right)+1}{q(q-1)}\,q^{2P}=0.
\end{align*}

Suppose now that $n\ge 6$. Similarly to the computation in the proof of Theorem~\ref{thm-quadrat:mor}, we obtain
\begin{align*}
    \frac{N(P+1)}{q-1}
    &=\frac{\left(q^{\frac{n}{2}}+1\right)q^{n-2}}{\left(q^{\frac{n}{2}-1}+q\right)(q-1)}\, q^{P(n-2)}+(-1)^P \frac{\left(q^{\frac{n}{2}-1}-1\right)q^{\frac{n}{2}}}{q^{\frac{n}{2}-1}+q}\,q^{\frac{n}{2}P},\\
    -\frac{(q+1)N(P)}{q-1}
    &=-\frac{\left(q^{\frac{n}{2}}+1\right)(q+1)}{\left(q^{\frac{n}{2}-1}+q\right)(q-1)}\, q^{P(n-2)}+(-1)^P \frac{\left(q^{\frac{n}{2}-1}-1\right)(q+1)}{q^{\frac{n}{2}-1}+q}\,q^{\frac{n}{2}P}\\
\intertext{and}
    \frac{qN(P-1)}{q-1}
    &=\frac{\left(q^{\frac{n}{2}}+1\right)q^{3-n}}{\left(q^{\frac{n}{2}-1}+q\right)(q-1)}\, q^{P(n-2)}+(-1)^P \frac{\left(q^{\frac{n}{2}-1}-1\right)q^{1-\frac{n}{2}}}{q^{\frac{n}{2}-1}+q}\,q^{\frac{n}{2}P},
\end{align*}
and hence, $\#\mor_P(\PP^1,X)(\fq)$ adds up to
\begin{align*}
    &\:\frac{\left(q^{\frac{n}{2}}+1\right)\left(q^{n-2}-q-1+q^{3-n}\right)}{\left(q^{\frac{n}{2}-1}+q\right)(q-1)}\, q^{P(n-2)}
    +(-1)^P \frac{\left(q^{\frac{n}{2}-1}-1\right)\left(q^{\frac{n}{2}}+q+1+q^{1-\frac{n}{2}}\right)}{q^{\frac{n}{2}-1}+q}\,q^{\frac{n}{2}P}\\
    =&\:\frac{\left(q^{\frac{n}{2}}+1\right)\left(q^{n-2}-1\right)\left(1-q^{3-n}\right)}{\left(q^{\frac{n}{2}-1}+q\right)(q-1)}\, q^{P(n-2)}
    +(-1)^P \frac{\left(q^{n-1}-q\right)\left(1+q^{-\frac{n}{2}}\right)}{q^{\frac{n}{2}-1}+q}\,q^{\frac{n}{2}P}\\
    =&\:\frac{\left(q^{\frac{n}{2}}+1\right)\left(q^{n-2}-1\right)\left(q^{n-3}-1\right)}{q^{n-2}\left(q^{\frac{n}{2}-2}+1\right)(q-1)}\, q^{P(n-2)}+(-1)^P \frac{\left(q^{n-2}-1\right)\left(q^{\frac{n}{2}}+1\right)}{q^{\frac{n}{2}}\left(q^{\frac{n}{2}-2}+1\right)}\,q^{\frac{n}{2}P}.
\end{align*}
\end{proof}

\subsection{Computation of \texorpdfstring{$N(P)$}{N(P)} if \texorpdfstring{$2\nmid n$}{2!|n}}
Lastly, we consider the case where $n$ is odd. Let $r\in\fqt$ be monic. According to Theorem~\ref{thm-S_r(f)}, we then have
\[
    S_r(f)=\varphi(r)\abs{r}^{\frac{n}{2}}
\]
if $r$ is a square in $\fqt$, and $S_r(f)=0$ otherwise. Since it follows from~{\cite[Proposition~1.7]{rosen2002}} that $\varphi(r^2)=\varphi(r)\abs{r}$, we can also write
\begin{align}
    S_{r^2}(f)
    &=\varphi(r^2)\abs{r^2}^{\frac{n}{2}}=\varphi(r)\abs{r}^{n+1}.\label{gl-nUNGERADE:Sr2(f)}
\end{align}
In particular, for all $k\in\N_0$ we have
\begin{align}
    S_{t^k}(f)
    &=\begin{dcases}
        1&\textup{if }k=0,\\
        \varphi(t^k)\abs{t^k}^{\frac{n}{2}}=\frac{q-1}{q}\,q^{k\left(\frac{n}{2}+1\right)}&\textup{if }k\ge 1\textup{ and }2\mid k,\\
        0&\textup{else}
    \end{dcases}\label{gl-nUNGERADE:S_tc}
\end{align}
because $t^k$ is a square in $\fqt$ if and only if $k$ is even.

As in both previous cases, where $n$ was even, we start by computing the first sum in~\eqref{gl-N(P)=1+2}.

\begin{lem}\label{lem-nUNGERADE:1.Summe}
Let $n$ be odd. For $n=3$ we have
\begin{align*}
    \sum_{\substack{r\in\fqt\\ \abs{r}\le q^{P-1}\\ r \textup{ monic}}}\frac{S_r(f)}{\abs{r}^n}\, I_r
    &=\begin{dcases}
           \frac{q^2-1}{2q}P q^P+\frac{1}{q}\,q^{P} &\textup{if }2\mid P,\\
    \frac{q^2-1}{2q}P q^P+\frac{q^2+1}{2q}\,q^{P}&\textup{if }2\nmid P,
    \end{dcases}
\end{align*}
and for $n\ge 5$ we have
\begin{align*}
    \sum_{\substack{r\in\fqt\\ \abs{r}\le q^{P-1}\\ r \textup{ monic}}}\frac{S_r(f)}{\abs{r}^n}\, I_r=
\begin{dcases} 
    \frac{q-q^{2-n}}{1-q^{3-n}}\,q^{P(n-2)} -\frac{\left(q^2-1\right)q^{n-3}}{q^{n-2}-q}\, q^{\frac{n-1}{2}P} &\textup{if }2\mid P,\\
    \frac{q-q^{2-n}}{1-q^{3-n}}\,q^{P(n-2)} -\frac{\left(q^2-1\right)q^{\frac{n-3}{2}}}{q^{n-2}-q}\, q^{\frac{n-1}{2}P}&\textup{if }2\nmid P.
    \end{dcases}
\end{align*}
\end{lem}

\begin{proof}
Let $r\in\fqt$ be monic with $\rho\coloneq \deg(r)\le P-1$. Since $S_{t^k}(f)=0$ for odd $k\in\N$ according to~\eqref{gl-nUNGERADE:S_tc}, we can write
\begin{align*}
    s_r\coloneq \sum_{k=0}^{P-\rho-1}q^{nk}S_{t^{P-\rho-k-1}}(f)
    &=\sum_{k=0}^{P-\rho-1}q^{n(P-\rho-k-1)}\,S_{t^k}(f)
    =\sum_{k=0}^{\floor{\frac{P-\rho-1}{2}} }q^{n(P-\rho-2k-1)}\,S_{t^{2k}}(f),
\end{align*}
and hence, it follows from~\eqref{gl-nUNGERADE:S_tc} that
\begin{align*}
    s_r&=\sum_{k=1}^{\floor{\frac{P-\rho-1}{2}} } q^{n(P-\rho-2k-1)}\,\frac{q-1}{q}\, q^{2k\left(\frac{n}{2}+1\right)}+q^{n(P-\rho-1)}\\
    &=\frac{q-1}{q^{n+1}}\, q^{nP-n\rho}\sum_{k=1}^{\floor{\frac{P-\rho-1}{2}}}\left(q^{2-n}\right)^k+\frac{1}{q^n}\,q^{nP-n\rho}.
\end{align*}
As $n\ge 3$, we have
\begin{align*}
    \sum_{k=1}^{\floor{\frac{P-\rho-1}{2}} }\left(q^{2-n}\right)^k 
    &= q^{2-n}\,\frac{1-\left(q^{2-n}\right)^{\floor{\frac{P-\rho-1}{2}}}}{1-q^{2-n}},
\end{align*}
which yields
\begin{align*}
    s_r
    &=\left(\frac{(q-1)q^{1-2n}}{1-q^{2-n}}+\frac{1}{q^n}\right) q^{n P-n\rho} - \frac{(q-1)q^{1-2n}}{1-q^{2-n}}\, q^{nP-n\rho+(2-n)\floor{\frac{P-\rho-1}{2}}}.
\end{align*}
It then follows
\begin{align*}
    I_r
    =\frac{q^{n\rho+n+1}}{q^{2P}}\cdot s_r
    &=\left(\frac{(q-1)q^{2-n}}{1-q^{2-n}}+q\right)q^{P(n-2)}- \frac{(q-1)q^{2-n}}{1-q^{2-n}}\,q^{P(n-2)+(2-n)\floor{\frac{P-\rho-1}{2}}} \\
    &=\frac{q-q^{2-n}}{1-q^{2-n}}\,q^{P(n-2)}- \frac{(q-1)q^{2-n}}{1-q^{2-n}}\,q^{P(n-2)+(2-n)\floor{\frac{P-\rho-1}{2}}}.
\end{align*}
To compute 
\[
    \sum_{\substack{r\in\fqt\\ \abs{r}\le q^{P-1}\\ r\textup{ monic}}}\frac{S_r(f)}{\abs{r}^n} \, I_r,
\]
we notice that as per~\eqref{gl-Quadrat:S_r(f)}, we only need to sum over all elements of the form $r^2\in\fqt$ for some monic $r\in\fqt$ with 
\[
    \abs{r}\le \left(q^{P-1}\right)^{\frac{1}{2}}=q^{\frac{P-1}{2}},
\]
i.e. $\abs{r}\le q^{\floor{\frac{P-1}{2}}}$, since $\deg(r)\in\N_0$. Thus, inserting~\eqref{gl-nUNGERADE:Sr2(f)} yields
\begin{align*}
    \sum_{\substack{r\in\fqt\\ \abs{r}\le q^{P-1}\\ r\textup{ monic}}}\frac{S_r(f)}{\abs{r}^n} \, I_r
    &=\sum_{\substack{r\in\fqt\\ \abs{r}\le q^{\floor{\frac{P-1}{2}}}\\ r\textup{ monic}}} \frac{\varphi(r)\abs{r}^{n+1}}{\abs{r^2}^n}\cdot \frac{q-q^{2-n}}{1-q^{2-n}}\,q^{P(n-2)}\\
    &\quad- \sum_{\substack{r\in\fqt\\ \abs{r}\le q^{\floor{\frac{P-1}{2}}}\\ r\textup{ monic}}} \frac{\varphi(r)\abs{r}^{n+1}}{\abs{r^2}^n}\cdot \frac{(q-1)q^{2-n}}{1-q^{2-n}}\,q^{P(n-2)+(2-n)\floor{\frac{P-\deg(r^2)-1}{2}}}.
\end{align*}
From
\begin{align*}
    q^{(2-n)\floor{\frac{P-\deg(r^2)-1}{2}}}
    &=q^{(2-n)\floor{\frac{P-1}{2}} -(2-n)\deg(r)}
    =q^{(2-n)\floor{\frac{P-1}{2}}}\abs{r}^{n-2}
\end{align*}
for all monic $r\in\fqt$, we then get
\begin{align}
    \sum_{\substack{r\in\fqt\\ \abs{r}\le q^{P-1}\\ r\textup{ monic}}}\frac{S_r(f)}{\abs{r}^n} \, I_r
    &=\frac{q-q^{2-n}}{1-q^{2-n}}\,q^{P(n-2)} \sum_{\substack{r\in\fqt\\ \abs{r}\le q^{\floor{\frac{P-1}{2}}}\\ r\textup{ monic}}} \frac{\varphi(r)}{\abs{r}^{n-1}}\notag\\
    &\quad- \frac{(q-1)q^{2-n}}{1-q^{2-n}}\,q^{P(n-2)+(2-n)\floor{\frac{P-1}{2}}}\sum_{\substack{r\in\fqt\\ \abs{r}\le q^{\floor{\frac{P-1}{2}}}\\ r\textup{ monic}}} \frac{\varphi(r)}{\abs{r}}.\label{gl-nUNGERADE:vormEinsetzen}
\end{align}
According to Lemma~\ref{lem-geometrischeReihe}, we have
\begin{align*}
    \sum_{\substack{r\in\fqt\\ \abs{r}\le q^{\floor{\frac{P-1}{2}}}\\ r\textup{ monic}}}\frac{\varphi(r)}{\abs{r}}
    &=-\frac{q-1}{1-q}\, q^{\floor{\frac{P-1}{2}}}
    =q^{\floor{\frac{P-1}{2}}}
\end{align*}
and
\begin{align*}
    \sum_{\substack{r\in\fqt\\ \abs{r}\le q^{\floor{\frac{P-1}{2}}}\\ r\textup{ monic}}} \frac{\varphi(r)}{\abs{r}^{n-1}}
    &=\begin{dcases}
        1+\frac{q-1}{q} \floor{\frac{P-1}{2}} &\textup{if }n=3,\\
        \frac{1-q^{2-n}}{1-q^{3-n}}-\frac{(q-1)q^{2-n}}{1-q^{3-n}}\,q^{\floor{\frac{P-1}{2}} \left(3-n\right)} &\textup{if }n\ge 5.
    \end{dcases}
\end{align*}
For $n=3$ we then obtain
\begin{align*}
    \sum_{\substack{r\in\fqt\\ \abs{r}\le q^{P-1}\\ r\textup{ monic}}}\frac{S_r(f)}{\abs{r}^n} \, I_r
    &=\frac{q-q^{-1}}{1-q^{-1}}\,q^{P} \left(1+\frac{q-1}{q}\floor{\frac{P-1}{2}}\right)- \frac{(q-1)q^{-1}}{1-q^{-1}}\,q^{P-\floor{\frac{P-1}{2}}}q^{\floor{\frac{P-1}{2}}}\\
     &=\left(q+1\right)\,q^{P} \left(1+\frac{q-1}{q}\floor{\frac{P-1}{2}} \right) - q^{P}\\
     &=\frac{q^2-1}{q}\floor{\frac{P-1}{2}} q^P+ q\cdot q^{P}
\end{align*}
by inserting into~\eqref{gl-nUNGERADE:vormEinsetzen}, and thus
\begin{align*}
    \sum_{\substack{r\in\fqt\\ \abs{r}\le q^{P-1}\\ r\textup{ monic}}}\frac{S_r(f)}{\abs{r}^n} \, I_r
    &=\begin{dcases} 
    \frac{q^2-1}{2q}(P-2) q^P+q\cdot q^{P} &\textup{if }2\mid P,\\
    \frac{q^2-1}{2q}(P-1) q^P+q\cdot q^{P} &\textup{if }2\nmid P
    \end{dcases}\\
    &=\begin{dcases}
    \frac{q^2-1}{2q}P q^P+\frac{1}{q}\,q^{P} &\textup{if }2\mid P,\\
    \frac{q^2-1}{2q}P q^P+\frac{q^2+1}{2q}\,q^{P}&\textup{if }2\nmid P.
    \end{dcases}
\end{align*}

Now we assume $n\ge 5$. Then, inserting into~\eqref{gl-nUNGERADE:vormEinsetzen} yields
\begin{align*}
      \sum_{\substack{r\in\fqt\\ \abs{r}\le q^{P-1}\\ r\textup{ monic}}}\frac{S_r(f)}{\abs{r}^n} \, I_r
    &=\frac{q-q^{2-n}}{1-q^{2-n}}\,q^{P(n-2)} \left(\frac{1-q^{2-n}}{1-q^{3-n}}-\frac{(q-1)q^{2-n}q^{\floor{\frac{P-1}{2}} (3-n)}}{1-q^{3-n}}\right)\notag\\
    &\quad- \frac{(q-1)q^{2-n}}{1-q^{2-n}}\,q^{P(n-2)+(2-n)\floor{\frac{P-1}{2}}} q^{\floor{\frac{P-1}{2}}}\\
    &=\frac{q-q^{2-n}}{1-q^{3-n}}\,q^{P(n-2)} -\frac{(q-1)q^{2-n}\left(q-q^{2-n}\right)}{\left(1-q^{2-n}\right)\left(1-q^{3-n}\right)}\,q^{P(n-2)+\floor{\frac{P-1}{2}} (3-n)}\notag\\
    &\quad- \frac{(q-1)q^{2-n}}{1-q^{2-n}}\,q^{P(n-2)+(3-n)\floor{\frac{P-1}{2}}}.
\end{align*}
Since we can compute
\begin{align*}
    &\:\frac{(q-1)q^{2-n}\left(q-q^{2-n}\right)}{\left(1-q^{2-n}\right)\left(1-q^{3-n}\right)} + \frac{(q-1)q^{2-n}}{1-q^{2-n}}\\
    =&\:\frac{(q-1)q^{2-n}}{1-q^{2-n}}\left(\frac{q-q^{2-n}}{1-q^{3-n}}+1\right)\\
    =&\:\frac{(q-1)q^{2-n}}{1-q^{2-n}}\cdot \frac{1+q-q^{2-n}-q^{3-n}}{1-q^{3-n}}\\
    =&\:(q-1)q^{2-n}\,\frac{q+1}{1-q^{3-n}}
    =\frac{q^2-1}{q^{n-2}-q},
\end{align*}
it follows from
\begin{align*}
    P(n-2)+\floor{\frac{P-1}{2}}(3-n)
    &=\begin{dcases}
        \frac{n-1}{2}P+n-3&\textup{if }2\mid P,\\
        \frac{n-1}{2}P+\frac{n-3}{2}&\textup{if }2\nmid P
    \end{dcases}
\end{align*}
that the coefficient of $q^{\frac{n-1}{2}P}$ is 
\[
    \begin{dcases} 
    -\frac{q^2-1}{q^{n-2}-q}\, q^{n-3}&\textup{if }2\mid P,\\
    -\frac{q^2-1}{q^{n-2}-q}\,q^{\frac{n-3}{2}}&\textup{if }2\nmid P.
    \end{dcases}
\]
In total, we then obtain
\begin{align*}
    \sum_{\substack{r\in\fqt\\ \abs{r}\le q^{P-1}\\ r\textup{ monic}}}\frac{S_r(f)}{\abs{r}^n} \, I_r
    &=\begin{dcases} 
    \frac{q-q^{2-n}}{1-q^{3-n}}\,q^{P(n-2)} -\frac{\left(q^2-1\right)q^{n-3}}{q^{n-2}-q}\, q^{\frac{n-1}{2}P} &\textup{if }2\mid P,\\
    \frac{q-q^{2-n}}{1-q^{3-n}}\,q^{P(n-2)} -\frac{\left(q^2-1\right)q^{\frac{n-3}{2}}}{q^{n-2}-q}\, q^{\frac{n-1}{2}P}&\textup{if }2\nmid P.
    \end{dcases}
\end{align*}  
\end{proof}

This then yields a formula for $N(P)$ for odd $n\ge 3$.

\begin{thm}\label{thm-nUngerade:N(P)}
Let $n$ be odd. For $n=3$ we have
\begin{align*}
   N(P)
   &=\begin{dcases}
   \frac{q^2-1}{2q} Pq^P+q^P &\textup{if }2\mid P,\\
   \frac{q^2-1}{2q}Pq^P+\frac{q^2+1}{2q}\, q^P &\textup{if }2\nmid P,
   \end{dcases}
\end{align*}
and for $n\ge 5$ we have
\begin{align*}
   N(P)&=
   \begin{dcases}
       \frac{q^{n-1}-1}{q^{n-2}-q} \, q^{P(n-2)}-(q-1)\frac{q^{n-2}+1}{q^{n-2}-q}\, q^{\frac{n-1}{2}P}&\textup{if }2\mid P,\\
       \frac{q^{n-1}-1}{q^{n-2}-q} \, q^{P(n-2)}-\frac{\left(q^2-1\right)q^{\frac{n-3}{2}}}{q^{n-2}-q}\, q^{\frac{n-1}{2}P}&\textup{if }2\nmid P.
   \end{dcases}
\end{align*}
\end{thm}

\begin{proof}
As in the proof of Theorem~\ref{thm-Quadrat:N(P)} and Theorem~\ref{thm-nUngerade:N(P)}, we only need to apply Lemma~\ref{lem-nUNGERADE:1.Summe} and compute the second sum in~\eqref{gl-N(P)=1+2}. Since squares in $\fqt$ are of even degree, we have $S_r(f)=0$ for all monic $r\in\fqt$ with $\deg(r)=P$ if $P$ is odd, as per~\eqref{gl-nUNGERADE:Sr2(f)}. Hence, in this case it follows directly from Lemma~\ref{lem-nUNGERADE:1.Summe} that
\begin{align*}
    N(P)
    &=\sum_{\substack{r\in\fqt\\ \abs{r}\le q^{P-1}\\ r\textup{ monic}}} \frac{S_r(f)}{\abs{r}^n}\, I_r
    =\begin{dcases}
        \frac{q^2-1}{2q}\,P q^P+\frac{q^2+1}{2q}\,q^P&\textup{if }n=3,\\
        \frac{q^{n-1}-1}{q^{n-2}-q}\,q^{P(n-2)} -\frac{\left(q^2-1\right)q^{\frac{n-3}{2}}}{q^{n-2}-q}\, q^{\frac{n-1}{2}P}&\textup{if }n\ge 5.
    \end{dcases}
\end{align*}

We now assume that $P$ is even. According to~\eqref{gl-nUNGERADE:Sr2(f)}, we then have
\begin{align*}
     \sum_{\substack{r\in\fqt\\ \abs{r}=q^P\\ r\textup{ monic}}} \frac{S_r(f)}{\abs{r}^n}\, I_r
    &= q^{P(n-2)} \sum_{\substack{r\in\fqt\\ \abs{r}=q^{\frac{P}{2}}\\ r\textup{ monic}}} \frac{S_{r^2}(f)}{\abs{r^2}^n}\\
    &=q^{P(n-2)} \sum_{\substack{r\in\fqt\\ \abs{r}=q^{\frac{P}{2}}\\ r\textup{ monic}}} \frac{\varphi(r)\!\left(q^{\frac{P}{2}}\right)^{n+1}}{\left(q^{\frac{P}{2}}\right)^{2n}}\\
    &=q^{\frac{n-3}{2}P} \sum_{\substack{r\in\fqt\\ \abs{r}=q^{\frac{P}{2}}\\ r\textup{ monic}}}\varphi(r). 
\end{align*}
From Lemma~\ref{lem-sum:phi}, we therefore get
\begin{align*}
    \sum_{\substack{r\in\fqt\\ \abs{r}=q^P\\ r\textup{ monic}}} \frac{S_r(f)}{\abs{r}^n}\, I_r
    &=q^{\frac{n-3}{2}P}\cdot \frac{q-1}{q}\left(q^{\frac{P}{2}}\right)^2=\frac{q-1}{q}\,q^{\frac{n-1}{2}P},
\end{align*}
which we can insert into~\eqref{gl-N(P)=1+2}. By applying Lemma~\ref{lem-nUNGERADE:1.Summe}, we can then compute $N(P)$. This gives
\begin{align*}
    N(P)
    &=\frac{q^2-1}{2q}\,P q^P+\frac{1}{q}\,q^{P} +\frac{q-1}{q}\, q^{P}
    =\frac{q^2-1}{2q}\,P q^P+q^{P}
\end{align*}
for $n=3$, and by computing
\begin{align*}
    \frac{(q^2-1)q^{n-3}}{q^{n-2}-q}-\frac{q-1}{q}
    &=(q-1)\frac{(q+1)q^{n-3}-\left(q^{n-3}-1\right)}{q^{n-2}-q}
    =(q-1)\frac{q^{n-2}+1}{q^{n-2}-q},
\end{align*}
we obtain
\begin{align*}
    N(P)
    &= \frac{q^{n-1}-1}{q^{n-2}-q}\,q^{P(n-2)} - \frac{(q^2-1)q^{n-3}}{q^{n-2}-q}\, q^{\frac{n-1}{2}P}+  \frac{q-1}{q}\, q^{\frac{n-1}{2}P}\\
    &=\frac{q^{n-1}-1}{q^{n-2}-q}\,q^{P(n-2)} - (q-1)\frac{q^{n-2}+1}{q^{n-2}-q}\, q^{\frac{n-1}{2}P}
\end{align*}
for $n\ge 5$.
\end{proof}

With this, we can finally prove Theorem~\ref{thm-nUNGERADE:mor}.

\begin{proof}[Proof of Theorem~\ref{thm-nUNGERADE:mor}]
Suppose first that $n=3$. According to Theorem~\ref{thm-nUngerade:N(P)} and~\eqref{gl-mor_P}, we have
\begin{align*}
    \#\mor_P(\PP^1,X)(\fq)
        &=\begin{dcases} 
        \frac{q+1}{2q}(P+1)q^{P+1}+\frac{q^2+1}{2q(q-1)}\, q^{P+1}-\frac{(q+1)^2}{2q}Pq^P\\\quad -\frac{q+1}{q-1}\,q^P
         +\frac{q+1}{2}(P-1)q^{P-1}+\frac{q^2+1}{2(q-1)}\,q^{P-1}&\textup{if }2\mid P,\\
         \frac{q+1}{2q}(P+1)q^{P+1}+\frac{1}{q-1}\, q^{P+1}-\frac{(q+1)^2}{2q}Pq^P\\\quad -\frac{(q+1)\left(q^2+1\right)}{2q(q-1)}\,q^P
        +\frac{q+1}{2}(P-1)q^{P-1}+\frac{q}{q-1}\,q^{P-1}&\textup{if }2\nmid P.
        \end{dcases}
\end{align*}
Since 
\begin{align*}
    \frac{q+1}{2q}\cdot q-\frac{(q+1)^2}{2q}+\frac{q+1}{2}\cdot\frac{1}{q}
    &=0,
\end{align*}
the coefficient of $Pq^P$ vanishes, independent of the parity of $P$. Hence, we obtain
\begin{align*}
    \#\mor_P(\PP^1,X)(\fq)
    &=\left(\frac{q+1}{2q}+\frac{q^2+1}{2q(q-1)}\right)q^{P+1}-\frac{q+1}{q-1}\,q^P+\left(-\frac{q+1}{2}+\frac{q^2+1}{2(q-1)}\right)q^{P-1}\\
    &=\frac{2q^3-2q^2-2q+2}{2q(q-1)}\,q^P
    =\frac{q^2-1}{q}\,q^P
\end{align*}
if $P$ is even, and
\begin{align*}
        \#\mor_P(\PP^1,X)(\fq)
        &=\frac{q(q^2-1)+2q^2-(q+1)(q^2+1)-(q^2-1)+2q}{2q(q-1)}\,q^P =0
    \end{align*}
if $P$ is odd. 

Now we assume that $n\ge 5$. Computing the summands in~\eqref{gl-mor_P}, we get
    \begin{align*}
        \frac{N(P+1)}{q-1}
        &=\begin{dcases}
            \frac{\left(q^{n-1}-1\right)q^{n-2}}{\left(q^{n-2}-q\right)(q-1)}\,q^{P(n-2)}-\frac{(q+1)q^{n-2}}{q^{n-2}-q}\,q^{\frac{n-1}{2}P}&\textup{if }2\mid P,\\
            \frac{\left(q^{n-1}-1\right)q^{n-2}}{\left(q^{n-2}-q\right)(q-1)}\,q^{P(n-2)}-\frac{(q^{n-2}+1)q^{\frac{n-1}{2}}}{q^{n-2}-q}\,q^{\frac{n-1}{2}P}&\textup{if }2\nmid P,
        \end{dcases}\\
        -\frac{(q+1)N(P)}{q-1}
        &=\begin{dcases}
        -\frac{(q^{n-1}-1)(q+1)}{\left(q^{n-2}-q\right)(q-1)}\,q^{P(n-2)}+\frac{(q^{n-2}+1)(q+1)}{q^{n-2}-q}\,q^{\frac{n-1}{2}P}&\textup{if }2\mid P,\\
            -\frac{(q^{n-1}-1)(q+1)}{(q^{n-2}-q)(q-1)}\,q^{P(n-2)}+\frac{(q+1)^2q^{\frac{n-3}{2}}}{q^{n-2}-q}\,q^{\frac{n-1}{2}P}&\textup{if }2\nmid P,
        \end{dcases}
        \intertext{as well as}
        \frac{qN(P-1)}{q-1}
        &=\begin{dcases}
            \frac{(q^{n-1}-1)q^{3-n}}{\left(q^{n-2}-q\right)(q-1)}\,q^{P(n-2)}-\frac{q+1}{q^{n-2}-q}\,q^{\frac{n-1}{2}P}&\textup{if }2\mid P,\\
            \frac{(q^{n-1}-1)q^{3-n}}{\left(q^{n-2}-q\right)(q-1)}\,q^{P(n-2)}-\frac{(q^{n-2}+1)q^{\frac{3-n}{2}}}{q^{n-2}-q}\,q^{\frac{n-1}{2}P}&\textup{if }2\nmid P.
        \end{dcases}
    \end{align*}
In both cases, we can compute
\begin{align*}
    \frac{q^{n-1}-1}{(q^{n-2}-q)(q-1)}\left(q^{n-2}-q-1+q^{3-n}\right)
    &=\frac{(q^{n-1}-1)(1-q^{2-n})}{q-1}
    =\frac{(q^{n-1}-1)(q^{n-2}-1)}{q^{n-2}(q-1)}
\end{align*}
for the coefficient of $q^{P(n-2)}$. 
Since
\begin{align*}
    &-\frac{(q+1)q^{n-2}}{q^{n-2}-q}+\frac{(q^{n-2}+1)(q+1)}{q^{n-2}-q}-\frac{q+1}{q^{n-2}-q}=0,
    \end{align*}
the coefficient of $q^{\frac{n-1}{2}P}$ vanishes if $P$ is even, and we thus obtain
\begin{align*}
    \#\mor_P(\PP^1,X)(\fq)
    =\frac{(q^{n-1}-1)(q^{n-2}-1)}{q^{n-2}(q-1)}\,q^{P(n-2)}
\end{align*}
in this case. If $P$ is odd, we have the same coefficient for $q^{P(n-2)}$ as in the case where $P$ is even, whereas for the coefficient of $q^{\frac{n-1}{2}P}$ we can compute
\begin{align*}
   &\:\frac{-(q^{n-2}+1)q^{\frac{n-1}{2}}+(q+1)^2 q^{\frac{n-3}{2}}-(q^{n-2}+1)q^{\frac{3-n}{2}}}{q^{n-2}-q}\\
   =&\:\frac{q^{\frac{3-n}{2}}\left(-q^{2n-4}-q^{n-2}+(q+1)^2q^{n-3}-q^{n-2}-1\right)}{q^{n-2}-q}\\
   =&\:-\frac{q^{\frac{3-n}{2}}\left(q^{n-1}-1\right)\left(q^{n-3}-1\right)}{q^{n-2}-q}
    =-q^{\frac{1-n}{2}}\left(q^{n-1}-1\right).
\end{align*}
In total, we then get
\begin{align*}
    \#\mor_P(\PP^1,X)(\fq)
    &=\frac{(q^{n-1}-1)(q^{n-2}-1)}{q^{n-2}(q-1)}\,q^{P(n-2)}-\frac{q^{n-1}-1}{q^{\frac{n-1}{2}}}\,q^{\frac{n-1}{2}P}
\end{align*}
if $P$ is odd.
\end{proof}

\printbibliography
\end{document}